\documentclass{amsart}

\usepackage{epsfig,amsmath,amsthm,amssymb, bbm, bm, xcolor, float, mathrsfs}
\usepackage[shortlabels]{enumitem}
\numberwithin{equation}{section}

\usepackage{calc}

\makeatletter
\newcommand{\dbwidetilde}[1]{{%
  \mathpalette\db@widetilde{#1}%
}}
\newcommand{\db@widetilde}[2]{%
  \sbox\z@{$\m@th#1\widetilde{#2}$}%
  \ht\z@=.8\ht\z@
  \widetilde{\box\z@}%
}
\def\beqs#1\eeqs{%
    \begin{equation}\begin{split}%
    #1%
    \end{split}\end{equation}%
}
\def\beqs#1\eeqs{%
    \begin{equation}\begin{split}%
    #1%
    \end{split}\end{equation}%
}
\makeatother
\newcommand{\lp}{\left(}
\newcommand{\rp}{\right)}
\newcommand{\lb}{\left[}
\newcommand{\rb}{\right]}

\newcommand{\bZ}{{\bf Z}}
\newcommand{\bY}{{\bf Y}}
\newcommand{\bpm}{\begin{pmatrix}}
\newcommand{\epm}{\end{pmatrix}}
\newcommand{\bbeta}{{\bm\beta}}
\newcommand{\bx}{{\bf x}}

\newcommand{\bg}{{\bf g}}

\newcommand{\balph}{{\bm\alpha}}
\newcommand{\bmta}{{\bm\theta}}
\newcommand{\indep}{\perp \!\!\! \perp}
\newtheorem{thm}{Theorem}[section]
\newtheorem{lemma}[thm]{Lemma}
\newtheorem{prop}[thm]{Proposition}

\theoremstyle{definition}
\newtheorem{example}[thm]{Example}
\newtheorem{defn}{Definition}
\newtheorem*{defn*}{Definition}
\theoremstyle{remark}
\newtheorem{remark}[thm]{Remark}

\newcommand{\beq}{\begin{equation}}
\newcommand{\eeq}{\end{equation}}

\newcommand{\noin}{\noindent}
\newcommand{\E}{\mathbb E}
\newcommand{\R}{\mathbb R}

\newcommand{\e}{\mathrm{exp}}
\newcommand{\ta}{\theta}
\newcommand{\sg}{\sigma}
\newcommand{\la}{\langle}
\newcommand{\ra}{\rangle}
\newcommand{\lla}{\left\langle}
\newcommand{\rra}{\right\rangle}
\definecolor{AfonsoBlue}{RGB}{30,65,123}

\newcommand{\anya}[1]{{\textcolor{cyan}{[\textbf{Anya: #1}]}}}

\title[Likelihood Maximization \& Moment Matching in low SNR GMMs]{Likelihood Maximization and Moment Matching in Low SNR Gaussian Mixture Models} 
\date{}
\author[A.E. Katsevich]{Anya E. Katsevich$^*$}
\thanks{$^*$Email: katsevich@cims.nyu.edu. Department of Mathematics, Courant Institute of Mathematical Sciences, New York University, USA.\\ AEK is supported by the DOE Computational Science Graduate Fellowship.}
\author[A.S.Bandeira]{Afonso S. Bandeira$^\dagger$}
\thanks{$^\dagger$Email: bandeira@math.ethz.ch. Department of Mathematics, ETH Zurich, Switzerland. \\Part of this work was done while ASB was with the Department of Mathematics, Courant Institute of Mathematical Sciences and the Center for Data Science at NYU and supported partly by NSF grants DMS-1712730, DMS-1719545 and by a grant from the Sloan foundation.}

\begin{document}

\maketitle

\begin{abstract}
We derive an asymptotic expansion for the log likelihood of Gaussian mixture models (GMMs) with equal covariance matrices in the low signal-to-noise regime. The expansion reveals an intimate connection between two types of algorithms for parameter estimation: the method of moments and likelihood optimizing algorithms such as Expectation-Maximization (EM). We show that likelihood optimization in the low SNR regime reduces to a sequence of least squares  optimization problems that match the moments of the estimate to the ground truth moments one by one. This connection is a stepping stone toward the analysis of EM and maximum likelihood estimation in a wide range of models. A motivating application for the study of low SNR mixture models is cryo-electron microscopy data, which can be modeled as a GMM with algebraic constraints imposed on the mixture centers. We discuss the application of our expansion to algebraically constrained GMMs, among other example models of interest.
\end{abstract}




\section{Introduction}
Gaussian mixtures are a useful model to describe data in a wide variety of applications. Nevertheless, strong theoretical guarantees on the performance of classical algorithms for inference in Gaussian mixture models (GMMs) are lacking. This is primarily due to the complicated structure of the GMM log likelihood landscape. The most popular algorithm for inference is Expectation-Maximization (EM), an iterative algorithm which performs ``soft assignment'' of observations to mixture components. Although EM maximizes a surrogate function to the log likelihood at each step, it can nevertheless be viewed as gradient ascent on the log likelihood in the setting we study here. As such, analyzing it and other log likelihood optimizing algorithms is challenging. 

Most existing guarantees are for for the case in which the component distributions of the mixture are ``well-separated''. 
In~\cite{balakrishnan2017} the authors characterize the basin of attraction in which the EM algorithm is guaranteed to converge to the global maximum of the likelihood of a well-separated two component mixture. This is generalized in~\cite{two-mixture}, in which the authors provide a global analysis of the convergence of EM for two component mixtures. A further generalization is obtained in~\cite{gradientEM}, in which the basin of attraction for gradient EM in arbitrary mixtures with equal covariances is quantified, also under the assumption of some separation between component distributions. For mixtures with three or more components, \cite{jin2016local} shows that there are well-separated mixtures for which the log likelihood landscape has bad local maxima. Moreover, they show that in some cases, the EM algorithm can converge to these bad critical points with high probability.

Other algorithms have been proposed for learning the parameters of poorly separated GMMs in polynomial time~\cite{GMMarbsep, two-mix-poor-sep} without relying on the log likelihood. The latter paper is based on the method of moments. While EM and its variants are the most widely used methods for inference in GMMs, the method of moments is another class of inference methods which bypasses the log likelihood entirely. This method was proposed by Karl Pearson in his 1894 paper \cite{pearson_crabs}, which also introduces the Gaussian mixture inference problem for the first time. Pearson shows that the parameters of a mixture of two one-dimensional Gaussians can be deduced from the mixture's first six moments. In general, the approach is to form estimates from the data of enough moments of the distribution to uniquely specify it. The challenge is then to ``invert'' the moments to recover the ground truth parameters. As an example, for a $K$-component uniform mixture with centers $\mu_{(1)},\dots,\mu_{(K)}\in\R^d$, which we collectively denote $\mu$, the moments are defined as \begin{align*} T_1(\mu) &=\frac1K\mu_{(1)} + \dots + \frac1K\mu_{(K)}\in\R^d, \\
T_2(\mu) &= \frac1K\mu_{(1)}\mu_{(1)}^T + \dots + \frac1K\mu_{(K)}\mu_{(K)}^T\in\R^{d\times d},\end{align*} with higher moments $T_k(\mu)$ given by higher order tensors. Given estimates of the ground truth moment tensors $T_k(\mu_*)$, moment inversion amounts to finding $\mu$ such that $T_k(\mu) = T_k(\mu_*), k=1,2,\dots$. In some models, the moment tensors take a particularly convenient form and can be inverted explicitly. When this is not possible, one alternative approach is to minimize the objective function
\beq\label{introeqn}\min_{\mu\in\R^{dK}}\sum_k \lambda_k\|T_k(\mu) - T_k(\mu_*)\|^2,\eeq where $\lambda_k$ are regularizing weights.

In this paper, we study the log likelihood landscape of Gaussian mixture models in $\R^d$ with the following defining characteristics: (1) the covariance matrices of the mixture components are all the same, and (2) the center of each mixture component is small in norm relative to $|\Sigma|^{1/d}$, where $\Sigma$ is the covariance of each of the mixture components. We will think of $\Sigma$ as being known (although this is not required for our main result), and the mixture centers as the ``signal'' we wish to estimate. Since this is made more difficult by larger variances, one can think of mixtures with this second feature as having low signal-to-noise ratio (SNR).

We show an intimate connection between log likelihood optimization and the method of moments in the low SNR regime. We do so by deriving an asymptotic series expansion of the GMM log likelihood with respect to a small parameter related to the SNR. This expansion illuminates the structure of the likelihood landscape.  It shows that in the low SNR regime, log likelihood maximization reduces to a sequence of least squares minimization problems, in which successively higher moments are matched to those of the true distribution on the manifold on which all previous moments have been been fixed to the ground truth values.  

For the uniform mixture example, these minimization problems take the form
\beq\label{introeqn-2}\min_{\mu\in\mathcal V_{k-1}}\|T_k(\mu) - T_k(\mu_*)\|^2, k=1,2,\dots,\eeq where $\mathcal V_0 = \R^{Kd}$ and $$\mathcal V_k = \{\mu\in\R^{Kd}\mid T_\ell(\mu) = T_\ell(\mu_*), \ell=1,\dots, k\}.$$

This is very similar to the strategy of moment inversion described above. Indeed, taking weights $\lambda_1\gg\lambda_2\gg\dots$ in \eqref{introeqn} effectively reduces that minimization problem to the sequence of individual moment matching problems \eqref{introeqn-2}. 

This connection allows one to relate the roughness of the log-likelihood landscape with the roughness of the landscape of least squares moment matching objectives. In Section~\ref{sec:examplesimplications}, we will classify the critical points of this moment matching landscape in two illustrative examples: a uniform mixture of two Gaussians in arbitrary dimension and an arbitrary (finite) mixture of Gaussians in one dimension. In general, however, understanding the roughness of this landscape can be a highly non-trivial task and is outside the scope of this paper. 

The motivation for Taylor expanding the log likelihood comes from~\cite{bandeira17_opt}. In that paper, Taylor expansions for upper and lower bounds on the log likelihood are derived.
However, in order to analyze algorithms which depend on the landscape of the log likelihood (i.e. on the function's derivatives), a Taylor expansion of the log likelihood itself is needed. For a certain class of models a recent paper \cite{fan2020likelihood}, fruit of parallel research efforts, also establishes such an expansion, as we will discuss in more detail below.

A natural class of models to study in the low SNR regime are algebraically structured mixture models. A prime example is 
the orbit retrieval model, also known as multi-reference alignment (MRA). In this class of models, a known algebraic constraint relates the centers of the mixture components to one another. Specifically, the centers are all determined from any one center by applying to it the elements of a subgroup of rotations on $\R^d$. In particular, the centers therefore all have the same norm. This class of models is motivated by problems arising in molecule imaging using Cryo-Electron Microscopy (cryo-EM). The goal is to infer the density of a molecule from noisy observations of it in different unknown orientations. At a first approximation the data can be described by a GMM in which the centers are constrained to be observations of the same (unknown) molecule from different viewing directions. We describe this model in more detail in Section~\ref{subsec:algmodel}. 

In a recent paper \cite{fan2020likelihood}\footnote{The authors learned of this work at an earlier stage of preparing the current manuscript, and have since leveraged insights of~\cite{fan2020likelihood} to help motivate and simplify some of our arguments. The derivation of the expansion in the case of general mixture models appears to require a different set of techniques and our arguments are quite different overall.}, the authors derive an asymptotic expansion for the log likelihood of the orbit retrieval model. Remarkably, the authors then leverage this expansion and the algebraic structure present in the orbit retrieval problem, to analyze the critical points of the log-likelihood landscape (via the critical points of the moment matching objectives \eqref{introeqn-2}). The expansion we derive in the more general context of GMMs reduces to that of \cite{fan2020likelihood} when the model is of the orbit retrieval type.
 While an analysis of the complexity of the moment-matching landscape in the general case is beyond the scope of this paper, the results in~\cite{fan2020likelihood} on the orbit retrieval model illustrate how such an analysis can be used to draw conclusions about the log likelihood landscape and maximum likelihood estimation. 

We note that our likelihood expansion applies to other important algebraically structured models as well, such as \emph{heterogeneous MRA}, in which the centers constitute the orbits of several points in $\R^d$ under a group action. Cryo-EM data can be modeled this way, since one often observes a molecule in several different conformations. The distinct orbits are then the rotations of these distinct conformations. 

The method of moments is a natural approach for inference in algebraically structured models in the low SNR regime, and a theoretical understanding of the method has been developed in this setting. With the help of our asymptotic expansion, we expect that some of this understanding can be transferred to draw conclusions about likelihood optimizing methods such as EM. 
We discuss this as well as potential implications of the expansion beyond algebraically structured models in Section~\ref{subsec:algmodel}.

We have alluded to the fact that in the model setting we study, EM is the same as gradient descent on the negative log likelihood. We make this precise in Section~\ref{sec:EM}. Specifically, we show that for finite mixtures and orbit retrieval models, both standard EM and a variant known as gradient EM, are given by gradient ascent on the log likelihood with respect to the centers of the mixture. This implies that an understanding of the likelihood landscape directly translates into an understanding of the fixed points of EM and their basins of attraction. However, we will also show that the standard EM algorithm is suboptimal in the low SNR regime, in that it corresponds to gradient descent with too small a step size. This was shown in~\cite{fan2020likelihood} for the orbit retrieval model. Thus gradient EM is a better option, since the step size is user-specified. 

We note that in order to use our expansion to draw conclusions about EM and maximum likelihood estimation, a finite sample analysis of the likelihood landscape is required. Here, we focus only on the population log likelihood. In~\cite{fan2020likelihood} concentration of the sample log likelihood and its first two derivatives around their population analogues is established for the orbit retrieval model. We also note that while our asymptotic expansion does not require the ground truth mixture weights to be known, we assume this is the case in our discussion of the consequences of the expansion. 
\subsection*{Acknowledgements} We would like to thank Jonathan Niles-Weed, Matthias Loeffler, and Justin Finkel for insightful discussions. We also thank Zhou Fan for pointing us to his paper.

\subsection*{Paper Organization} The paper is organized as follows. In Section~\ref{sec:main}, we introduce the general class of GMMs we will consider and some guiding example models. We then state our main result, the asymptotic expansion of the GMM log-likelihood in the low SNR regime.  In Section~\ref{sec:EM}, we show that for this class of GMMs, EM is the same as gradient ascent on the log likelihood with respect to the centers. We also apply the asymptotic expansion to draw conclusions about the EM algorithm and its variants in the low SNR regime. In Section~\ref{sec:examplesimplications}, we apply the expansion to several example models to draw conclusions about critical points of the corresponding log likelihood landscapes. We also discuss the implications of the expansion for models with algebraic structure motivated by the cryo-EM problem. In Section~\ref{sec:expansion} we present the proof of the expansion, deferring technical parts to the appendix. 

\subsection*{Notation}\label{subsec:notation}
For $x\in\R^d$, we let $g(x)$ denote the probability distribution function of the standard normal Gaussian $\mathcal N(0, I)$, 
$$g(x) = (\sqrt{2\pi})^{-d}e^{-\|x\|^2/2}.$$ For a set $K\subset\R^d$ and a point $x\in\R^d$, we define $K-x = \{y-x\mid y\in K\}$.  If $K$ is compact we define $\|K\|_{\infty}: = \sup_{x\in K}\|x\|$, where $\|x\|$ denotes the Euclidean norm of $x$.

For a probability measure $\rho$ on $\R^d$, we write $\mathrm{supp}(\rho)$ to denote its support. We write $\bmta\sim\rho$ to denote that $\bmta\in\R^d$ is a random variable with distribution $\rho$ (bold-font letters will always denote random variables). 
For $\bmta\sim\rho$ with $\rho$ compactly supported, we define $$\|\bmta\|_\infty =\|\rho\|_\infty = \|\mathrm{supp}(\rho)\|_{\infty}.$$ 

For the moment tensors of $\bmta\sim\rho,$ we write
$$T_k(\rho) = T_k(\bmta) = \E_{\bmta\sim\rho}\lb\bmta^{\otimes k}\rb = \int_{\R^d}x^{\otimes k}\rho(dx)\in\lp\R^d\rp^{\otimes k},$$ $k=1,2,3,\dots$. We use $T_{1:k}$ as shorthand for $T_1, \dots, T_k$. For two tensors $T, S\in\lp\R^d\rp^{\otimes k}$ with real entries, we let $\la T, S\ra$ denote the entry-wise inner product of their vectorizations in $\R^{d^k}$, and $\|T\| = \la T, T\ra^{1/2}.$

\section{Model Description and Main Theorem}\label{sec:main}






Let $Y\in\R^d$ be distributed according to a Gaussian mixture in which the component distributions have the same, nondegenerate covariance $\Sigma$. The assumption of equal covariances allows us to write $Y$ as a Gaussian perturbation $\Sigma^{\frac12}Z$ of a random variable $\bmta\in\R^d$ encoding the centers of the mixture components and the mixture weights. For example, if $Y$ is a uniform mixture of $K$ Gaussian distributions $\mathcal N(\mu_{j},\Sigma),\,j=1,\dots,K$, then $\bmta$ is a discrete random variable taking the value $\mu_{j}$ with probability $1/K$, $j=1,\dots, K$. In general, we have:
\beqs\label{mixture}Y&= \Sigma^{\frac12}Z + \bmta,\quad \bmta\sim\rho,\\
 Z&\sim\mathcal N(0, I),\;Z\indep\bmta.\eeqs If $\rho$ is a sum of point masses, then $Y$ is a discrete mixture of component distributions. If $\rho$ has a density, then $Y$ is a continuous mixture. 

We will consider maximum likelihood estimation of $\rho=\rho_*$ given independent identically distributed observations $y_i\sim Y,\;i=1,\dots, N$ in the case $N\to\infty$. The asymptotic expansion of the log likelihood presented in the next section is valid for the family of compactly supported measures $\rho$, and we therefore present it in this most general setting. Importantly, this general setting also includes the parametric framework in which it is known that $\rho_*$ belongs to a set parameterized by a finite number of variables. 

Note that if $\Sigma$ is known, then we can transform \eqref{mixture} into a mixture of spherical distributions by multiplying $Y$ by $\Sigma^{-\frac12}$. Thus, the case in which the component distribution covariances are known, equal, and nondegenerate, is equivalent to the model \beq\label{GMM}Y= \sg Z + \theta,\quad\bmta\sim\rho,\, Z\sim\mathcal N(0, I),\, Z\indep\bmta.\eeq We therefore assume the covariance is $\sg^2I$ from now on. (We do not set $\sg=1$ because it will be convenient to perform Taylor expansions in $1/\sg$). 

Now, the distribution $\rho$ induces a density $q_{\rho}(y)$ on $Y$. To compute $q_{\rho}$, note that
\beqs
\mathbb P(Y\in A) = \E_{\bmta\sim\rho}\lb\mathbb P(\sg Z + \bmta\in A\mid\bmta)\rb = \int_{A}\E_{\bmta\sim\rho}\lb g\lp\sg^{-1}(y-\bmta)\rp\rb dy.
\eeqs
This gives
\beqs
q_{\rho}(y) &= \E_{\bmta\sim\rho}\lb g(\sg^{-1}(y-\bmta))\rb\\
&=  (2\pi\sg)^{-\frac d2}\E_{\bmta\sim\rho}\lb \e\,\lp-\frac{\|y-\bmta\|^2}{2\sg^2}\rp\rb.
\eeqs
The population log likelihood $L(\rho;\;\rho_*)$ is then given by
\beqs
L(\rho;\,\rho_*) &= \E_{Y\sim q_{\rho_*}}\log\, q_{\rho}(Y) \\
&=\E_{Y\sim q_{\rho_*}}\log\,\E_{\bmta\sim\rho}\lb \e\,\lp-\frac{\|Y-\bmta\|^2}{2\sg^2}\rp\rb,\eeqs
 where we have discarded the normalization constant. Writing $Y=\sg Z + \bmta_*,\;\bmta_*\sim\rho_*$, we can also express the log likelihood in the following form:
 \beq
 L(\rho;\,\rho_*) = \E_{Z,\bmta_*\sim\rho_*}\log\,\E_{\bmta\sim\rho}\lb \e\,\lp-\frac{\|\sg Z + \bmta_*-\bmta\|^2}{2\sg^2}\rp\rb.
 \eeq
Abusing notation, we will sometimes write $L(\bmta;\,\bmta_*)$ for $L(\rho;\,\rho_*)$.  Note that $\rho=\rho_*$ is the unique global maximizer (up to measure zero) of $L$ in the space of probability distributions on $\R^d$. This is a consequence of the fact that $L(\rho;\;\rho_*) = -D_{\mathrm{KL}}(q_{\rho_*}||q_{\rho}) + \mathrm{const.},$ where $D_{\mathrm{KL}}$ is the Kullback-Leibler divergence between $q_{\rho_*}$ and $q_{\rho}$ and the constant term depends on $\rho_*$ only.

The GMM formulation \eqref{GMM} lends itself to the signal processing viewpoint of the statistical estimation problem. Namely, one can consider the observations $y_i$ as draws from the ``signal'' distribution $\rho$ corrupted by the additive noise $\sg Z$. This reasoning, as well as the likelihood expansion in the following section, motivate us to define the signal-to-noise ratio (SNR) as follows:
 \begin{defn}
Let $\rho$ be a compactly supported measure on $\R^d$. We define the SNR as
$$\mathrm{SNR}(\rho, \sg) =\|\mathrm{supp}(\rho) - T_1(\rho)\|^2_{\infty}/\sg^2.$$ 
\end{defn}
We note that this definition of SNR is not sensitive to how $\|\ta\|$ varies for $\ta\in\mathrm{supp}(\rho_*)$. For example, consider a discrete distribution $\rho_*$ concentrated on $\pm\ta_{1*},\pm\ta_{2*},$ where $\|\ta_{1*}\|\ll \|\ta_{2*}\|$. Then $\mathrm{SNR}(\rho_*,\sg)=\|\ta_{2*}\|/\sg$.  One could argue that the SNR should depend not just on $\|\ta_{2*}\|/\sg$ but also on how small $\|\ta_{1*}\|$ is relative to $\|\ta_{2*}\|$. 

However, we will see that for our purposes this is a natural definition of SNR. Indeed, it is is the scale parameter which emerges in the asymptotic expansion. The smaller this value, the more clear-cut the separation between successive moment-matching stages, as will be explained in Section~\ref{sec:mainresults}.


\subsection*{Guiding Examples}
It is helpful to keep in mind the following two classes of GMMs as examples of models to which the log likelihood expansion can be applied.  Both classes (i.e. families of measures $\rho$) can be parameterized by a finite number of variables, and we write the SNR and moments as functions of these parameters.\\

\noin\textbf{Discrete Finite Mixture Model.} This class of models can be described by $Y=\sg Z + \bmta\in\R^d$ where $\bmta\sim\rho$, a finite sum of point masses. In other words, $\rho$ is of the form 
\beq\rho(dx) = \sum_{j=1}^K\alpha_j\delta(x-\ta_j),\quad \ta_j\in\R^d,\alpha_j > 0, j=1,\dots, K, \;\sum_j\alpha_j=1.\eeq We have 
\beqs
T_k(\ta,\alpha) &= \sum_{j=1}^K\alpha_j\ta_j^{\otimes k}, k=1,2,\dots,\\
\mathrm{SNR}(\ta,\alpha, \sg) &= \max_{j=1,\dots, K}\|\ta_j - T_1(\ta,\alpha)\|^2/\sg^2,
\eeqs where $\ta,\alpha$ are shorthand for $(\ta_j,\alpha_j)_{j=1}^K$.  \\ \\

\noin\textbf{Orbit Retrieval.} Let $G\subset O(d)\subset \R^{d\times d}$ be a possibly infinite subgroup of the group of orthogonal rotations in $\R^d$. Let $\gamma$ be a measure on $G$, and ${\bf g}\in G$ denote the random variable with distribution $\gamma$. In the orbit retrieval model, we have  
\beq\label{orbitretrieval}Y= \sg Z + \bmta, \quad \mathrm{where}\;\bmta = {\bf g}\ta, \quad {\bf g}\sim\gamma,\; \ta\in\R^d.\eeq  
Here, ${\bf g}\ta$ denotes the action of ${\bf g}$ on $\ta$, in this case multiplication by a matrix. Note that $\ta\in\R^d$ is deterministic. 

In general, both the point $\ta$ whose orbit under $G$ constitutes the centers of the GMM, and the distribution $\gamma$, can be unknown. We have
\beqs
T_k(\ta, \gamma) &= \E_{{\bf g}\sim\gamma}\lb ({\bf g}\ta)^{\otimes k}\rb,\;k=1,2,\dots,\\
\mathrm{SNR}(\ta,\gamma,\sg)&=\sup_{g\in G}\|g\ta - T_1(\ta,\gamma)\|^2/\sg^2.
\eeqs

The term orbit retrieval is also sometimes used to denote the model in which $\gamma$ is known and given by the Haar measure (the uniform distribution on $G$).  Due to the invariance of the Haar measure under the action of $G$, we have $T_1= gT_1$ for any $g\in G$, so that the SNR is given by $\mathrm{SNR}(\ta,\gamma=\mathrm{Haar},\sg)=\|\ta - T_1\|^2/\sg^2.$

An example of a discrete orbit retrieval model is Multireference alignment (MRA). Here, $G= \{g_0,\dots, g_{d-1}\}$ is the group which acts on vectors in $\R^d$ by cyclically shifting their entries. In other words, we have
$$(g_j\ta)_k = \ta_{j+k\mod d},\quad j,k=0,\dots, d-1.$$ The measure $\gamma$ is therefore a sum of point masses, and induces the following distribution on $\bmta$:
$$\bmta\sim\rho,\quad \rho(dx) = \sum_{j=1}^K\gamma_j\delta(x-g_j\ta).$$


As an example of a continuous mixture, consider rotations in $\R^2$, distributed uniformly over angles of rotation $\omega\in[0,2\pi)$. Then the random variable $\bmta\in\R^2$ is distributed as
$$\bmta =\bpm \cos\omega & -\sin\omega\\ \sin\omega & \cos\omega\epm\ta, \; \omega\sim\mathrm{Unif}[0,2\pi), \quad\ta\in\R^2.$$

\subsection{Main results}\label{sec:mainresults}

In this section we state our main result, the asymptotic expansion of the log likelihood function. Recall that the log likelihood is given by
 \beq\label{Lrhorhostar}
 L(\bmta;\,\bmta_*,\sg) = \E_{Z,\bmta_*\sim\rho_*}\log\,\E_{\bmta\sim\rho}\lb \e\,\lp-\frac{\|\sg Z + \bmta_*-\bmta\|^2}{2\sg^2}\rp\rb.
 \eeq
\begin{thm}\label{mainthmsimple} Let $\bmta\sim\rho$ and $\bmta\sim\rho_*$ be compactly supported random variables on $\R^d$ and define $\delta =\delta(\rho, \rho_*)= \max\{\|\ta-\ta_*\|\mid \ta\in \mathrm{supp}(\rho), \ta_*\in\mathrm{supp}(\rho_*)\}.$ Let $m$ be a positive integer. If $$T_k(\bmta)=T_k(\bmta_*),\quad k=1, \dots, m-1$$ then for any $\sg>0$ we have:
\beqs\label{maineqnsimple}
-L(\bmta;\,\bmta_*,\sg) = &C_m(\bmta_*) + \sg^{-2m}\frac{1}{2(m!)}\|T_m(\bmta) - T_m(\bmta_*)\|^2  + \epsilon_m,
\eeqs
where $C_m(\bmta_*)$ is independent of $\bmta$ and
the error term $\epsilon_m = \epsilon_m(\bmta, \bmta_*)$ is bounded above by \beq\label{error}
|\epsilon_m(\bmta, \bmta_*)| \leq (m+1)! \lp C\frac{\delta}{\sg}\rp^{2m+2}\lp 1\vee \frac{\delta}{\sg}\rp^{2m+2},\eeq
where $C$ is a $d$-dependent absolute constant. 
\end{thm}

From \eqref{Lrhorhostar} it is clear that $L(\bmta;\,\bmta_*,\sg) = L(\bmta-c;\,\bmta_*-c,\sg)$ for any constant $c$. Note that $\delta$ is also invariant to shifts of $\bmta$ and $\bmta_*$ by the same amount. Thus,~\eqref{maineqnsimple} remains true if we substitute $\bmta-c,\bmta_*-c$ on the right hand side. However, the size of $\|T_m(\bmta) - T_m(\bmta_*)\|$ is on the order $\delta\lp\|\bmta\|_\infty\vee\|\bmta_*\|_\infty\rp^{m-1}$.  It is therefore not invariant to shifts. It will be desirable for $\delta$ to be of the same scale as $\|\bmta\|_\infty\vee\|\bmta_*\|$. In order to accomplish this, we will replace $\bmta$ and $\bmta_*$ by $\bmta-T_1^*$ and $\bmta_*-T_1^*$, respectively. From now on we will let $\bmta$, $\bmta_*$ denote these shifted random variables (i.e. assume $T_1^*=0$). 

We also have $L(\bmta;\,\bmta_*,\sg) = L(\lambda\bmta;\,\lambda\bmta_*,\lambda\sg)$ for any $\lambda >0$.  We will therefore set $\|\bmta_*\|_\infty  = 1$ in addition to assuming $T_1^*=0$. The ground truth SNR is then given by $\mathrm{SNR}(\rho_*, \sg) = 1/\sg^2$ and  the low SNR regime is characterized by $\sg\to\infty$. Note that we have the upper bound $$\delta \leq\|\bmta\|_{\infty}\vee 1.$$

\noin\textbf{Discussion.} 
Suppose the GMM lies in a parameterizable family $$\{Y=\sg Z + \bmta,\;\bmta\sim\rho_\ta\mid\ta\in\Theta\},$$ with $\Theta$ a set in a finite dimensional space. This allows us to consider gradient based local search algorithms for likelihood optimization in $\Theta$. Theorem~\ref{mainthmsimple} shows that in the low SNR regime $\sg\to\infty$, any such algorithm attempts to match the moments of $\rho = \rho_{\ta}$ to those of $\rho_* = \rho_{\ta_*}$ one by one, starting from the first moment. In other words, likelihood optimization reduces to the sequence of minimization problems 
\beq\label{momentmatchingI}\min_{\ta\in\mathcal V_{m-1}}\|T_{m}(\ta) - T_{m}(\ta_*)\|^2, \; m=1,2,3,\dots,\eeq where $\mathcal V_0 = \Theta$ and $\mathcal V_k\subset \Theta, k=1,2,\dots$ are the varieties
\beq\label{varietiesI}\mathcal V_k = \{\ta\in\Theta\mid T_\ell(\ta) = T_\ell(\ta_*),\;l=1,\dots, k\}.\eeq
This is a consequence of the fact that there is a scale separation between $\|T_m - T_m^*\|^2/\sg^{2m}$ and $\epsilon_m$. Indeed, provided $\|\bmta\|_\infty=\mathcal O(1)$ relative to $\sg$, the former is on the order $\sg^{-2m}$ and the latter is on the order $\sg^{-2m-2}$. 

Consider \eqref{maineqnsimple} when $m=1$. Due to this scale separation, the  algorithm will prioritize  minimization of $\|T_1(\ta) - T_1(\ta_*)\|^2$ over that of $\epsilon_1$. If the minimization is successful, $\ta$ will reach the variety $\mathcal V_1$. On this variety, the objective function to be minimized is now $\|T_2(\ta) - T_2(\ta_*)\|^2/4\sg^4$ to highest order. The algorithm will continue to step through these distinct minimization stages for $m=1,2,\dots$, provided it does not get stuck in a local minimum or saddle point $\ta$ of $\|T_m(\ta) - T_m(\ta_*)\|^2\big\vert_{\mathcal V_{m-1}}$, i.e. a critical point for which $T_m(\ta)\neq T_m(\ta_*)$.

This suggests an intimate connection between likelihood optimizing algorithms such as EM and the method of moments in the low SNR regime. 
The connection between these two classes of algorithms will be discussed further in Section~\ref{sec:examplesimplications}. 

For nonparametric GMMs in which there is no knowledge of $\rho_*$ beyond the compact support assumption, the asymptotic expansion of the log likelihood reduces to a sequence of minimization problems in the space of measures, i.e. $$\min_{\rho\in\mathcal V_{m-1}}\|T_m(\rho) - T_m(\rho_*)\|^2,$$ where $\mathcal V_{m-1}$ is defined analogously to the parametrizable case. We note that the moments are linear in $\rho$, so that the objective function is quadratic and the varieties are given by linear constraints. The sequence of least squares moment matching problems is therefore a quadratic programming problem, albeit in an infinite dimensional space. While an analysis of the non-parametric setting is outside the scope of this paper, it would be interesting to explore the connection between the method of moments and maximum likelihood estimation in this context. For results on maximum likelihood estimation and inference in non-parametric mixture models, see e.g.~\cite{saha2020,feng2018,laird78}.\\

Theorem~\ref{mainthmsimple} is a direct consequence of the following key Lemma. To state it we will need the following two definitions.
\begin{defn} Let $$T_{k_i} = \E_{\bmta\sim\rho_i}\lb\bmta^{\otimes k_i}\rb, \quad i=1,\dots, n,$$ i.e. $T_{k_i}$ is the order $k_i$ moment tensor of some distribution $\rho_i$. Consider $$T = \bigotimes_{i=1}^n T_{k_i}=T_{k_1}\otimes T_{k_2}\otimes\dots\otimes T_{k_n}.$$ We define the \textbf{total moment order} of $T$ to be $\sum_{i=1}^n k_i,$ i.e. the sum of all moment orders. We also say that the total moment order of each entry of $T$ is $\sum_{i=1}^n k_i$; in other words, the total moment order of products of entries of moment tensors is the sum of all moment orders in the product. 
\end{defn}
Let $\rho$ and $\rho_*$ be compactly supported measures on $\R^d$. In the definition and lemma below, we write $T_k, T_k^*$ as shorthand for $T_k(\rho), T_k(\rho_*)$, respectively. 
\begin{defn}\label{VkRk}
We define $V_k\lb T_{1:m}, T_{1:n}^*\rb$ as the set of all constant coefficient linear combinations of outer products of moment tensors $T_j, j\leq m,\, T_\ell^*,\ell\leq n$, of total moment order $k$.\\ \\
We define $R_k\lb T_{1:m}, T_{1:n}^*\rb$ as the set of all constant coefficient linear combinations of products of entries of moment tensors $T_j, j\leq m,\, T_\ell^*,\ell\leq n$, of total moment order $k$.
\end{defn}
\begin{lemma}\label{mainthm} Let $\rho$ and $\rho_*$ be compactly supported probability measures on $\R^d$.  For all $m=1,2,3,\dots$ we have
\beqs\label{maineqn}
-L(\rho;\,\rho_*) = &C(\rho_*) + \frac12\|T_1-T_1^*\|^2\sg^{-2} \\
&+ \sum_{k=2}^m\lp\frac{1}{2(k!)}\|T_k - T_k^*\|^2 + \lla T_k, Q_k\rra + r_k\rp \sg^{-2k} + \epsilon_m,
\eeqs
where $C(\rho_*)$ is independent of $\rho$, and 
\begin{align*}
Q_k &= Q_k(T_{1:k-1}, T_{1:k-1}^*)\in V_{k}\lb T_{1:k-1}, T_{1:k-1}^*\rb,\\
r_k &= r_k(T_{1:k-1}, T_{1:2k}^*)\in R_{2k}\lb T_{1:k-1}, T_{1:2k}^*\rb.
\end{align*} 
Moreover, $Q_k$ is such that
$$Q_k(T_{1:k-1}^*, T_{1:k-1}^*) = 0.$$ The error term $\epsilon_m = \epsilon_m(\rho, \rho_*)$ is the same as in \eqref{maineqnsimple}. 
\end{lemma}
The expansion \eqref{maineqn} generalizes the log likelihood series expansion (4.10) of \cite{fan2020likelihood}, which is specific to the orbit recovery model \eqref{orbitretrieval} in which the measure $\gamma$ on the group $G$ is the Haar (uniform) measure.  We note that our error bound decays as $(1/\sg)^{2m+2}$ when $\|\bmta\|_\infty = \mathcal O(1)$ as $\sg\to\infty$; this is a somewhat tighter bound than that of \cite{fan2020likelihood}, in which the error is shown to decay as $(\log\sg/\sg)^{2m+2}$ when $\sg\to\infty$ and $\|\ta\|/\sg = o(1/\log\sg)$. (Note that $\|\bmta\|_{\infty} = \|\ta\|$ for the orbit retrieval model). 

The expansion (4.10,\cite{fan2020likelihood}) is the same as \eqref{maineqn} except that (4.10) has no term of the form $\la T_k, Q_k\ra$. The following proposition explains why this is so. For the proof, see Proposition~\ref{app:orbit-recovery} in the appendix. 
\begin{prop} Let $\ta,\ta_*\in\R^d$, and $G \subset O(d)\subset \R^{d\times d}$ be a group. Define the random variable ${\bf g}\in G$ distributed according to ${\bf g}\sim\gamma,$ where $\gamma$ is the Haar measure on $G$. Let $T_k, T_k^*$ be the moment tensors of the distributions ${\bf g}\ta,\;{\bf g}\ta_*,\;{\bf g}\sim\gamma$, i.e.
\beq\label{orbit-dist}T_k = \E_{{\bf g}\sim\gamma}\lb\lp{\bf g}\ta\rp^{\otimes k}\rb,\quad T_k^* = \E_{{\bf g}\sim\gamma}\lb\lp{\bf g}\ta_*\rp^{\otimes k}\rb.\eeq Then for every tensor $Q \in V_k\lb T_{1:m}, T_{1:n}^*\rb$, we have
$$\lla T_k, Q\rra \in R_{2k}\lb T_{1:m}, T_{1:n}^*\rb.$$ In particular, the inner product $\la T_k, Q\ra$ depends only on moment tensors $T_j,\,j=1\dots, m$ even if $k > m$. 
\end{prop}
\noin The proof relies crucially on the Haar property of $\gamma$, namely, that ${\bf g}\stackrel{d}{=} h{\bf g}\;\forall h\in G.$

It follows from the proposition that for the orbit retrieval model, the $\sg^{-2k}$ coefficient (for $k>1$) in the asymptotic expansion \eqref{maineqn} of $-L(\ta;\;\ta_*)$ is given by
$$\frac{1}{2(k!)}\|T_k - T_k^*\|^2 + \lla T_k, Q_k\rra + r_k = \frac{1}{2(k!)}\|T_k - T_k^*\|^2 + \tilde r_k,$$ where $\tilde r_k = \lla T_k, Q_k\rra + r_k\in R_{2k}\lb T_{1:k-1},T_{1:2k}^*\rb.$

\section{Expectation Maximization As Gradient Descent}\label{sec:EM}
In this section, we consider the EM algorithm for finite GMMs and the orbit retrieval model, assuming that the mixture weights are known. We show that in these cases, both the standard and gradient EM algorithms reduce to gradient descent on the negative log likelihood with respect to the centers. This equivalence has been pointed out in the literature, in the context of particular models (see, for example,~\cite{wu2019,fan2020likelihood}). In light of the structure of the log likelihood landscape given in Theorem~\ref{mainthmsimple}, we show that the gradient descent step size of standard EM is unnecessarily small, leading to slow convergence. 

To present the EM algorithm, it will be helpful to slightly reformulate the model.

\subsection{Model Reformulation}\label{subsec:modelreform} We will represent mixture models by $$Y = \sg Z + \ta_{\bm\chi},\;\bm\chi\sim\rho,$$ where $\bm\chi$ is a latent membership variable defined on a set $X$ which parameterizes the component distributions of the mixture. We will use $\chi\in X$ (non bold) to denote a sample of $\bm\chi$. \\
\textbf{Finite Mixture Model.} We have $X=\{1,\dots, K\}$ and $\bm\chi\sim\rho$, where $\rho(d\chi) = \sum_{j=1}^K\alpha_j\delta(\chi - j)$, assumed known. We let $\ta = (\ta_\chi)_{\chi=1}^K$ denote the $K$ centers in $\R^d$. \\
\textbf{Orbit Retrieval} Let $G\subset O(d)\subset\R^{d\times d}$ be a possibly infinite group with elements $\{{\bf g}_\chi\mid\chi\in X\}$, and $\bm\chi\sim\rho$, arbitrary. We let $\ta\in\R^d$ denote the vector which generates all the centers $\ta_\chi$ through the action of $G$, i.e. $\ta_\chi = {\bf g}_\chi\ta$, $\chi\in X$. \\

Since both of these models are parameterized by $\ta$, we denote the density of $Y$ by $q_\ta$. It is given by
\beqs
q_{\ta}(y) = \E_{\bm\chi\sim\rho}\lb g\lp\frac{y-\ta_{\bm\chi}}{\sg}\rp\rb=  \int g\lp\frac{y-\ta_\chi}{\sg}\rp\rho(d\chi).
\eeqs 
In the next section we will need the conditional distribution $\chi\vert Y$. It is given by $$q_\ta(d\chi\mid y) = w_{\ta}(y, \chi)\rho(d\chi),$$ 
where we have defined $$w_{\ta}(y,\chi) =g\lp\frac{y-\ta_\chi}{\sg}\rp\bigg/ \E_{\bm\chi\sim\rho}\lb g\lp\frac{y-\ta_{\bm\chi}}{\sg}\rp\rb.$$ Finally, the log likelihood is given by
\beqs
L(\ta;\;\ta_*) =  \E_{Y\sim q_{\ta_*}}\log\,\E_{\bm\chi\sim\rho}\lb g\lp\frac{Y-\ta_{\bm\chi}}{\sg}\rp\rb,
\eeqs where we have discarded the normalization constant. 


\subsection{Algorithm Description}\label{subsec:EM}
Assume $\ta_*$ is the ground truth parameter. Define the function $Q(\ta'\mid\ta;\;\ta_*)$, which is a surrogate for the log likelihood. It is defined as follows:
\beqs
Q(\ta'\mid\ta;\;\ta_*) &=  \E_{Y\sim q_{\ta_*}}\E_{\bm\chi\sim q_{\ta}(\cdot\mid Y)}\log\,g\lp\frac{Y-\ta'_{\bm\chi}}{\sg}\rp\\
&= -\frac{1}{2\sg^2}\E_{Y\sim q_{\ta_*}}\int \|Y-\ta'_\chi\|^2w_\ta(Y, \chi)\rho(d\chi).
\eeqs
Note that if $\ta$ is an estimate of the ground truth parameter $\ta_*$, then the distribution $q_{\ta}(d\chi\mid Y) = w_\ta(Y,\chi)\rho(d\chi)$ is our best guess for the distribution of the latent membership variable $\chi$ given the observed data $Y$. 

Given an initialization $\ta^{(0)}$, the standard and gradient EM updates are given by
\beq\label{maxQ}\begin{aligned}
\ta^{(t+1)} &=\arg\max_{\ta'}Q(\ta'\mid \ta^{(t)};\;\ta_*) & \mathrm{(standard}\;\mathrm{EM)}&\\
\ta^{(t+1)} &= \ta^{(t)} + \tau\nabla_{\ta'}Q(\ta'\mid \ta^{(t)};\;\ta_*)\big\vert_{\ta'=\ta^{(t)}} & \mathrm{(gradient}\;\mathrm{EM)}&,
\end{aligned}\eeq
where $\tau>0$ is some step size. Solving the optimization problem for the standard EM update, we have for finite GMMs the update
\beq\label{EM-GMM-standard}\ta^{(t+1)}_\chi=\frac{\E_Y\lb w_{\ta^{(t)}}(Y,\chi)Y\rb}{\E_Y\lb w_{\ta^{(t)}}(Y, \chi)\rb},\quad \chi=1,\dots, K\eeq and for the orbit retrieval model
$$\ta^{(t+1)} = \int \E_Y\lb w_{\ta^{(t)}}(Y, \chi)g_{\chi}^{-1}Y\rb \rho(d\chi).$$

\begin{prop}
We have 
$$\nabla_{\ta'}Q(\ta'\mid \ta;\;\ta_*)\big\vert_{\ta'=\ta} = \nabla_{\ta}L(\ta;\;\ta_*)$$ for both the finite mixture and orbit retrieval models. Therefore, gradient based EM with step size $\tau$ is the same as gradient ascent on $L(\ta;\;\ta_*)$ with step size $\tau$. 

For the finite mixture model, the standard EM update can be written as
\beq\label{standardEM}\ta_\chi^{(t+1)} =\ta_\chi^{(t)}+ \tau_\chi^t\nabla_{\ta_\chi}L(\ta^{(t)};\;\ta_*),\quad \tau_\chi^t=\frac{\sg^2}{\alpha_\chi\E_Y\lb w_{\ta^{(t)}}(Y, \chi)\rb},\eeq for $\chi=1,\dots, K$. For the (possibly infinite) orbit retrieval model, the standard EM update can be written as
$$\ta^{(t+1)} =\ta^{(t)}+ \sg^2\nabla_{\ta}L(\ta^{(t)};\;\ta_*).$$
\end{prop}
We remark that in standard EM for finite mixtures, the step size $\tau_\chi^t$ varies with time, and is also different for different centers $\ta_\chi$.
\begin{proof} For the finite mixture, we use the fact that 
\beqs
\nabla_{\ta_\chi}\log\,&\E_{\bm\chi\sim\rho}\lb g\lp\frac{Y-\ta_{\bm\chi}}{\sg}\rp\rb \\&= -\frac{\alpha_\chi}{2\sg^2}\nabla_{\ta_\chi}\|Y-\ta_\chi\|^2g\lp\frac{Y-\ta_{\bm\chi}}{\sg}\rp\bigg/\E_{\bm\chi\sim\rho}\lb g\lp\frac{Y-\ta_{\bm\chi}}{\sg}\rp\rb\\
&= -\frac{\alpha_\chi}{2\sg^2}\nabla_{\ta_\chi}\|Y-\ta_\chi\|^2w_\ta(Y,\chi).
\eeqs

Thus, 
\beqs\nabla_{\ta_\chi}L(\ta;\;\ta_*) &=  -\frac{\alpha_\chi}{2\sg^2}\E_Y\lb\nabla_{\ta_\chi}\|Y-\ta_\chi\|^2w_\ta(Y,\chi)\rb,
\\ \nabla_{\ta'_\chi}Q(\ta'\mid \ta;\;\ta_*) &=  -\frac{\alpha_\chi}{\sg^2}\E_Y\lb \nabla_{\ta'_\chi}\|Y-\ta'_\chi\|^2w_\ta(Y, \chi)\rb.\eeqs 

For the orbit retrieval model, we use that $\|Y-\bg_{\chi}\ta\|  =\|\bg_{\chi}^{-1}Y - \ta\|$, so that
\beqs
\nabla_{\ta}\log\,\E_{\bm\chi\sim\rho}\lb g\lp\frac{Y-\ta_{\bm\chi}}{\sg}\rp\rb &= \nabla_{\ta}\log\,\E_{\bm\chi\sim\rho}\lb g\lp\frac{\bg_{\bm\chi}^{-1}Y - \ta}{\sg}\rp\rb
\\&=-\frac{1}{2\sg^2}\E_{\bm\chi\sim\rho}\lb w_\ta(Y, \bm\chi)\nabla_\ta \|\bg_{\bm\chi}^{-1}Y - \ta\|^2\rb.
\eeqs
Using this property to compute $\nabla_\ta Q$ as well, we obtain
\beqs\nabla_{\ta}L(\ta;\;\ta_*) &= -\frac{1}{2\sg^2}\E_Y\E_{\bm\chi\sim\rho}\lb w_\ta(Y, \bm\chi)\nabla_\ta \|\bg_{\bm\chi}^{-1}Y - \ta\|^2\rb\\
\nabla_{\ta'}Q(\ta'\mid \ta;\;\ta_*)&= -\frac{1}{2\sg^2}\E_Y\E_{\bm\chi\sim\rho}\lb w_\ta(Y, \bm\chi)\nabla_{\ta'}\|\bg_{\bm\chi}^{-1}Y - \ta'\|^2\rb.\eeqs
We immediately see that in both cases the gradients of $Q$ and $L$ are equal if $\ta'=\ta$. 

To see why standard EM is also gradient ascent on $L$, note that $Q$ is a quadratic function in the $\ta'_\chi$ for finite GMMs, and quadratic in $\ta'$ for orbit retrieval. Now, for a quadratic function $f(\ta') = -\frac c2\|\ta'\|^2 +x^T\ta' + \mathrm{const.}$ we can reach the global maximum in one step of gradient ascent from any point $\ta'$ by taking a step size $\frac1c$. In other words, $\ta'+\frac1c\nabla f(\ta')$ is the global maximizer of $f$.  Taking $\ta' = \ta^{(t)}$, we have 
$$\arg\max_{\ta'} Q(\ta'\mid\ta^{(t)};\;\ta_*) = \ta^{(t)}+\frac1c\nabla Q(\ta^{(t)}\mid\ta^{(t)};\;\ta_*) =\ta^{(t)}+\frac1c\nabla L(\ta^{(t)};\;\ta_*).$$
It remains to compute $c$. For the finite mixture, considering $Q$ as a function of $\ta_\chi$ we see that $$c = \frac{\alpha_\chi}{\sg^2}{\E_Y\lb w_{\ta^{(t)}}(Y, \chi)\rb}.$$ For orbit retrieval, we have
$$c = \frac{1}{\sg^2}\E_Y\int\rho(d\chi)w_{\ta^{(t)}}(Y, \chi) = \frac{1}{\sg^2}.$$
\end{proof}

\subsection{EM in Low SNR Regime}
We will use the expansion \eqref{maineqn} to informally demonstrate that in the low SNR regime, the step size in the standard EM update \eqref{standardEM} for the finite mixture model is much smaller than necessary, leading to slow convergence. The same is true for the orbit retrieval model, as shown in \cite{fan2020likelihood}.  

Let $\ta_* = (\ta_{1*}, \dots, \ta_{K*})$ be the centers of the ground truth model with $\ta_{j*}\in\R^d, j=1,\dots, K$ and $\ta = (\ta_{1}, \dots, \ta_{K})$ be the argument to the log likelihood. We define $\|\ta\|_\infty=\max_{j=1,\dots, K}\|\ta_{j}\|$. Recall that the ground truth mixture weights $\alpha_j$ are considered known, and $T_k(\ta) = \sum_{j=1}^K\alpha_j\ta_j^{\otimes k},k=1,2,\dots.$ As in Section~\ref{sec:main}, we will assume $T_1(\ta_*)=0$, $\|\ta_*\|_\infty=1$ and $\sg\gg1$. 

Now, we will consider the standard EM update in the direction of $T_1(\ta)$ and in the subspace orthogonal to it. First, we have
\begin{prop}\label{T1-EM} Let $\ta\mapsto G(\ta)$ be the standard EM update, given by \eqref{EM-GMM-standard}. Fix a constant $R>0$. Then for all $\ta$ such that $\|\ta\|_\infty/\sg\leq R,$ we have
$$\|T_1(G(\ta))\| \leq C(\|\ta\|_\infty\vee1)^2/\sg,$$ where $C$ depends on $d$ and $R$ only.
\end{prop} 
The proof is give in Proposition~\ref{app:T1-EM} in the appendix. Proposition~\ref{T1-EM} shows that if the EM iterates $\ta^{(t)}$ remain in a radius $\mathcal O(1)$ ball, then starting with $t=1$ the estimated first moment $T_1\lp\ta^{(t)}\rp$ is order $\mathcal O(\sg^{-1})$ away from $T_1^*=0$.

While $T_1$ nearly converges in one iteration of standard EM, the algorithm is much slower in the subspace orthogonal to $T_1$. To show this, we use the gradient descent representation of EM, 
$$\ta_\chi^{(t+1)} =\ta_\chi^{(t)}+ \tau_\chi^t\nabla_{\ta_\chi}L(\ta^{(t)};\;\ta_*),\quad\chi=1,\dots,K.$$ For $\ta$ such that $\|\ta\|_\infty=\mathcal O(1)$ with respect to $\sg$, we have 
\beq\label{LL-EM}-L(\ta;\;\ta_*)  = \mathrm{const.} +\frac12\|T_1(\ta)\|^2\sg^{-2} + q_4(\ta, \ta_*)\sg^{-4} + \mathcal O(\sg^{-6}),\eeq where $q_4$ is a homogeneous polynomial of order $4$ with respect to the entries of $\ta,\ta_*$. This follows from the representation of the log likelihood given in \eqref{maineqn}. Now, consider the gradient of \eqref{LL-EM} in the subspace orthogonal to $T_1$. On this subspace, the highest order term of $L$, given by $\|T_1(\ta)\|^2/\sg^2$, is constant (not optimized), while $q_4$ and its $\ta$-derivatives are order $\mathcal O(\sg^{-4})$. It follows that the optimal step size for gradient descent is $\mathcal O(\sg^{4})$. However, the actual step size is $\tau_\chi^t = \frac{\sg^2}{\alpha_\chi}\E_Y\lb w_{\ta^{(t)}}(Y, \chi)\rb^{-1} = \mathcal O(\sg^2)$, using that $\E_Y\lb w_{\ta^{(t)}}(Y, \chi)\rb=1+\mathcal O(\sg^{-2})$. This is shown in Lemma~\ref{lem:weights-expansion} of the appendix. 

Recall that for the orbit recovery model, the standard EM update is a gradient descent step on $-L$ with step size $\sg^2$ exactly. Numerical experiments in \cite{fan2020likelihood} show that gradient descent on $-L$ in the subspace orthogonal to $T_1$ with step size $\mathcal O\lp\sg^{4}\rp$ achieves much faster convergence than standard EM.

\section{Examples of Interest and Implications}\label{sec:examplesimplications} Recall that Theorem~\eqref{mainthmsimple} shows that in the low SNR regime, likelihood optimization for parameterizable GMMs reduces to the sequence of minimization problems
\beq\label{momentmatching}\min_{\ta\in\mathcal V_{k-1}}\|T_{k}(\ta) - T_{k}(\ta_*)\|^2, \; k=1,2,\dots,\eeq where $\mathcal V_0 = \Theta$ and
\beq\label{varieties}\mathcal V_k = \{\ta\in\Theta\mid T_\ell(\ta) = T_\ell(\ta_*),\;l=1,\dots, k\}, \quad k=1,2,\dots.\eeq


In the following two sections, we characterize the critical points of the minimization problems~\eqref{momentmatching} for two GMMs: a uniform mixture of two Gaussians in $\R^d$ and an arbitrary finite mixture of Gaussians in $\R$. We conclude the section with a discussion of the implications of the expansion for models with algebraic structure and GMMs with randomly chosen centers. We also discuss the necessary steps to make rigorous the connection between the moment matching and likelihood landscapes. 


\subsection{Uniform Mixture of Two Gaussians in $\R^d$}\label{subsec:mixtwo}
Let $Y=\sg Z + \bmta\in\R^d$, where $\bmta\sim\rho$, which belongs to the family
$$\{\rho(dv) = \frac12\delta(v-\ta_1) + \frac12\delta(v-\ta_2)\mid \ta_1,\ta_2\in\R^d\}.$$ 
Motivated by ~\cite{two-mixture}, we study the moment matching minimization problems in the following coordinates: 
\beq\label{alphabet}\alpha = \frac12\ta_1 + \frac12\ta_2,\quad \beta =\ta_1 - \ta_2.\eeq Define $\alpha_*,\beta_*$ analogously for the ground truth parameters. This is a natural reparameterization for the landscape, since
\beq
T_1(\rho) = \alpha,\quad T_2(\rho) = \alpha\alpha^T + \frac14\beta\beta^T.
\eeq
We see that the first moment $T_1^*$ determines $\alpha_*$, while $\beta_*$ is determined up to sign from $T_2^*$ given that $\alpha = \alpha_*$. Swapping $\ta_1$ and $\ta_2$ does not change the mixture distribution (since it is uniform), so $\alpha$ and $\pm\beta$ uniquely specify the distribution. It therefore suffices to consider the first two moment-matching optimization problems. The first unconstrained optimization problem, $\min_{\alpha\in\R^d}\|\alpha-\alpha_*\|^2$ has only a global minimum at $\alpha=\alpha_*$. Now, on the manifold $\alpha = \alpha_*$, the second optimization problem reduces to $$\min_{\beta\in\R^d}\|\beta\beta^T - \beta_*\beta_*^T\|^2.$$ We see that the points $\beta = \pm \beta_*$ are global minima, while $\beta=0$ is a saddle point.

This aligns with the results of ~\cite{two-mixture} on the fixed points of EM. The authors reformulate the EM updates in the $\alpha, \beta$ coordinates \eqref{alphabet}. Letting $\alpha^{(t)}, \beta^{(t)},t=0,1,\dots$ be the EM iterates, they show that $\alpha^{(t)}$ converges to $\alpha_*$ as $t\to\infty$, while $\beta^{(t)}$ converges to $\pm\beta_*$ if $\la \beta^{(0)}, \beta_*\ra\neq 0$ and $\beta^{(t)}$ converges to $0$ if $\la \beta^{(0)}, \beta_*\ra=0$. 

\subsection{Mixture of $K$ Gaussians in $\R$}
Let $Y = \sg Z + \bmta\in\R$, where $\bmta\in\R$ is distributed according to $\rho$ in the family
$$\mathcal R_{\alpha\text{-mix}} =\{\rho_\ta(dx)=\sum_{j=1}^K\alpha_j\delta(x-\ta_j)\mid \ta=(\ta_1, \dots, \ta_K)\in\R^K\}.$$ Here, $\alpha_j$ are positive weights summing to $1$. They are assumed known, so that the unknown parameters are $\ta=(\ta_1, \dots, \ta_K)\in\R^K$. Interestingly, if the mixture is uniform ($\alpha_j=1/K\;\forall j$), then in the low SNR regime this model is equivalent to the following orbit retrieval model studied in \cite{fan2020likelihood}:
$Y = \sg Z + \bmta\in\R^K$, where $\bmta\in\R^K$ is distributed according to $\nu$ in the family
$$ \mathcal R_{\mathrm{orbit}} = \{\nu_\ta(dx)=\frac{1}{K!}\sum_{j=1}^{K!}\delta(x-{\bf g}_j\ta)\mid \ta\in\R^K\}.$$ Here, $G = \{{\bf g}_j,\; j=1, \dots, K!\}\subset O(K)\subset \R^{K\times K}$ is a subgroup of the orthogonal group acting on vectors in $\R^K$ by permuting their entries. In other words, the orbit of $\ta$ under $G$ is the set of all permutations of the entries of $\ta$. 

The two models are equivalent in the sense that there is a one-to-one mapping 
\beq\label{onetoone}\{T_{1:k}(\rho)\mid\rho\in \mathcal R_{1/K\text{-mix}}\} \longleftrightarrow \{T_{1:k}(\rho)\mid\rho\in \mathcal R_{\mathrm{orbit}}\}.\eeq To show this, define the polynomials $p_\ell(\ta) = \frac1K\sum_{j=1}^K\ta_j^\ell$, so that $T_\ell(\rho_\ta) = p_\ell(\ta)$ for $\rho_\ta\in \mathcal R_{1/K\text{-mix}}$. Now, let $\nu_\ta$ be the corresponding measure in $\mathcal R_{\mathrm{orbit}}$. The entries of the tensor $T_\ell(\nu_\ta)$ are $\ell$-degree polynomials in $\R\lb \ta_1, \dots, \ta_K\rb$ which are invariant under permutation of the $\ta_j$. But the polynomials $p_{j},j=1,\dots, \ell$ generate the permutation invariant polynomials of degree at most $\ell$ (see \cite{fan2020likelihood} and the references therein), showing that both sets in \eqref{onetoone} are in one-to-one correspondence with $\{(p_1(\ta),\dots, p_k(\ta))\mid \ta\in \R^K\}$. 

In particular, \cite{fan2020likelihood} shows that the moment-matching problem for the orbit retrieval model reduces to 
$$\min_{x\in\mathcal V_k} (p_{k+1}(\ta) - p_{k+1}(\ta_*))^2,$$ where $\mathcal V_k = \{\ta\in\R^K\mid p_j(\ta) = p_j(\ta_*), j=1,\dots, k\}.$ This is precisely the moment-matching problem for a uniform mixture on $\R$.

We now generalize results in \cite{fan2020likelihood} on critical points of the above moment matching landscape to the case of non-uniform mixtures in $\R$. Fix positive weights $\alpha = (\alpha_1, \dots, \alpha_K)$ summing to $1$, and define $$p_\ell(\ta) = \sum_{j=1}^K\alpha_j\ta_j^\ell,$$ so that $p_\ell(\ta) = T_\ell(\rho_\ta)$ for a distribution $\rho_\ta\in\mathcal R_{\alpha\text{-mix}}$ on $\R$. 
For a fixed $\ta_* = (\ta_1^*, \dots, \ta_K^*)$ let $\mathcal V_k$ be the variety $\mathcal V_k = \{\ta\in\R^K\mid p_j(\ta) = p_j(\ta_*), j=1,\dots, k\}.$ 

The following result characterizes critical points of the moment matching objective function that are not global minima. 
\begin{prop}\label{mixture-in-R}The following holds for any generic $\ta_*$: Define $f_{n+1}:\R^K\to\R$ by $$f_{n+1}(\ta) = \frac12(p_{n+1}(\ta) -p_{n+1}^*)^2,$$ where $p_{n+1}^* = p_{n+1}(\ta_*)$. Then 
\begin{enumerate}[(a)]
\item A point $x= (x_1, \dots, x_K)\in\mathcal V_n\setminus \mathcal V_{n+1}$ is a critical point of $f_{n+1}\vert_{\mathcal V_n}$ if and only if it is a critical point of $p_{n+1}\vert_{\mathcal V_n}$, if and only if exactly $n$ coordinates $x_j$ are distinct.  
\item Let $x\in\mathcal V_n\setminus \mathcal V_{n+1}$ be a critical point of $f_{n+1}\vert_{\mathcal V_n}$. Assume without loss of generality that $x_1 > x_2 > \dots > x_n$ are the distinct centers, and let $(m_1, \dots, m_n)$ be the multiplicity vector, i.e. $m_i$ is the number of times $x_i$ repeats. We have the following classification of $x$: 
\begin{itemize}
\item If the multiplicity vector has the form $(m_1, 1, m_3, 1, \dots)$ then $x$ is a local minimum of $f_{n+1}\vert_{\mathcal V_n}$ if $p_{n+1}(x)> p_{n+1}^*$ and a local maximum if $p_{n+1}(x)<p_{n+1}^*$.
\item If the multiplicity vector has the form $(1, m_2, 1, m_4, \dots)$, then $x$ is a local minimum of $f_{n+1}\vert_{\mathcal V_n}$ if $p_{n+1}(x)< p_{n+1}^*$ and a local maximum if $p_{n+1}(x)>p_{n+1}^*$. 
\item If the multiplicity vector is not of either form, then $x$ is a saddle point of $f_{n+1}\vert_{\mathcal V_n}$ and of $p_{n+1}\vert_{\mathcal V_n}$ 
\end{itemize}
\item There are no local minima of $f_{n+1}\vert_{\mathcal V_n}$ on $\mathcal V_n\setminus \mathcal V_{n+1}$ if the weights $\alpha_j$ are uniform.
\end{enumerate}
\end{prop}
\begin{example}\textbf{No local minima of $f_2$ on $\mathcal V_1$}\\
Suppose $x = (x_1,\dots, x_K)$ is a critical point of $f_2\vert_{\mathcal V_1}$ such that $x\notin \mathcal V_2$. This implies $x_1=\dots=x_K = p_1^*$, i.e. $m_1=K>1$. But then $p_2(x) = (p_1^*)^2 < p_2^*$ (recall that $p_1^*, p_2^*$ are the first and second moments of the distribution $\rho_{x^*}$, respectively), so $x$ is a local maximum. 
\end{example}
\begin{proof}[Proof of Proposition~\ref{mixture-in-R}]
A point $x\in \mathcal V_n$ is a critical point of $f_{n+1}\vert_{\mathcal V_n}$ if and only if $\nabla f_{n+1}(x)$ lies in the span of $\nabla  p_j(x), j=1,\dots, n$. We have $$\nabla f_{n+1} = (p_{n+1}- p_{n+1}^*)\nabla p_{n+1},$$ so if $x\notin \mathcal V_{n+1}$ then $p_{n+1}(x) - p_{n+1}^*\neq 0$, implying $\nabla p_{n+1}(x)$ also lies in the span of $\nabla p_j(x), j=1,\dots, n$. Hence $x$ is a critical point of $p_{n+1}\vert_{\mathcal V_n}$. 

Now, by arguments analogous to those in Lemma 4.23 of \cite{fan2020likelihood}, every point in $\mathcal V_n$ has at least $n$ distinct entries (for generic $x_*$), and $\mathcal V_n$ is nonsingular (i.e. the gradients $\nabla p_j(x), j=1,\dots, n$ are linearly independent for every $x\in\mathcal V_n$).

We show that a critical point $x$ of $p_{n+1}\vert_{\mathcal V_n}$ can have at most $n$ distinct entries. Note that $\partial_j p_k(x) = k\alpha_jx_j^{k-1}$. Since $\nabla p_{n+1}(x)$ lies in the span of the gradients $\nabla p_k(x)$, there exist $\lambda_1, \dots, \lambda_n$ such that 
\beq\label{vander1}\bpm \alpha_1(n+1)x_1^n\\ \vdots \\ \alpha_K(n+1)x_K^n\epm = \lambda_0\bpm \alpha_1\\ \vdots \\ \alpha_K\epm +\lambda_{1}\bpm 2\alpha_1x_1^{1}\\ \vdots \\ 2\alpha_Kx_K^{1}\epm +  \dots + \lambda_{n-1}\bpm n\alpha_1x_1^{n-1}\\ \vdots \\ n\alpha_Kx_K^{n-1}\epm.\eeq Define the polynomial \beq\label{pn}q_n(x) = (n+1)x^n - (n\lambda_{n-1}x^{n-1}+\dots + 2\lambda_1x+\lambda_0).\eeq Now, $q_n$ is an $n$th order polynomial, and \eqref{vander1} gives that $q_n(x_1) = \dots = q_n(x_K) = 0$ (since the $\alpha_j$ are nonzero). This implies that there are at most $n$ distinct points among $x_1,\dots,x_K$.

The second assertion follows from \cite{arnold}, but we provide a proof for the sake of completeness. We will use the following characterization of critical points on manifolds, reviewed in Appendix~\ref{app:miscellany}: \\

\noin Let $f:\R^K\to \R$ and $\mathcal M\subset \R^K$ be the intersection of level sets of functions $g_1, \dots, g_n$. Let $x\in\mathcal M$ be a critical point of $f\vert_{\mathcal M}$ and $c_1, \dots, c_n\in\R$ be such that \beq\label{critpoint}\nabla f(x) = \sum_{j=1}^nc_j\nabla g_j(x).\eeq Then $x$ is a saddle, local minimum, or local maximum of $f$ on $\mathcal M$ iff the quadratic form
\beq\label{critpoint2} \nabla^2 f(x) - \sum_{j=1}^nc_j\nabla^2g_j(x)\eeq is indeterminate, positive definite, or negative definite, respectively, on the tangent plane to $\mathcal M$ at $x$. \\

\noin We apply this result with $f=f_{n+1}$, $g_j=p_j$ and $\mathcal M = \mathcal V_n$. Let $x\in\mathcal V_n\setminus \mathcal V_{n+1}$ be a critical point of $f_{n+1}\vert_{\mathcal V_n}$. Without loss of generality, assume $x_1>x_2 >\dots > x_n$ are the distinct points. Letting $\lambda_j, j=1,\dots, n$ be as in \eqref{vander1}, we have $$\nabla f_{n+1}(x) = (p_{n+1}(x)- p_{n+1}^*)\sum_{j=1}^n\lambda_j\nabla p_j,$$ so that $c_j = (p_{n+1}(x)- p_{n+1}^*)\lambda_j$. Now the Hessian of $f_{n+1}$ is given by
\beq
\nabla^2 f_{n+1} = (p_{n+1} - p_{n+1}^*)\nabla^2 p_{n+1}+ \nabla p_{n+1}\nabla^T p_{n+1},
\eeq
where $\nabla p_{n+1}$ is a column vector. Since $\nabla p_{n+1}(x)$ is a linear combination of $\nabla p_j(x), j=1, \dots, n$, it is orthogonal to vectors in the tangent plane of $\mathcal V_{n+1}$ at $x$. We therefore drop it from the quadratic form and consider
\beqs
(p_{n+1} - p_{n+1}^*)\nabla^2 p_{n+1} &- \sum_{j=1}^n(p_{n+1} - p_{n+1}^*)\lambda_j\nabla^2p_j(x)\\
&=  (p_{n+1} - p_{n+1}^*)\lp\nabla^2 p_{n+1}(x) - \sum_{j=1}^n\lambda_j\nabla^2p_j(x)\rp\\
&= (p_{n+1} - p_{n+1}^*)\mathrm{diag}\big(\alpha_1q_n'(x_1),\dots, \alpha_Kq_n'(x_K)\big),
\eeqs
where the polynomial $q_n$ is as in \eqref{pn}. We now characterize the vectors $v\in\R^K$ in the tangent plane to $\mathcal V_n$ at $x$, i.e. perpendicular to $\nabla p_j(x), j=1,\dots, n$.  First, define the vectors $u_k\in \R^n,\,k=1,\dots,n$  by $$u_k = (x_1^{k-1},\dots, x_n^{k-1})^T, \; k=1,\dots, n.$$ Then the matrix $V\in\R^{n\times n}$ with columns $u_1, \dots, u_n$ is a Vandermonde matrix with determinant
$$\mathrm{det}(V) = \prod_{1\leq i < j\leq n}(x_i - x_j) \neq 0.$$ Thus, $u_1, \dots, u_n$ are linearly independent. Now, let $v$ be perpendicular to $\nabla p_j(x), j=1,\dots, n$. For such a $v$, we have
\beq\label{vm}
0 = \lla {v}, \nabla p_k\rra = k\sum_{i=1}^nx_i^{k-1}\sum_{\substack{j\;\mathrm{s.t.}\\x_j = x_i}}\alpha_jv_j = k\lla u_k, \tilde v\rra,
\eeq
where we have defined $\tilde v = (\tilde v_1,\dots, \tilde v_n)\in\R^n$ by $$\tilde {v}_{i} = \sum_{\substack{j\;\mathrm{s.t.}\\x_j = x_i}}v_j\alpha_j,\; i=1,\dots,n.$$ 
Taking $k=1,\dots,n$ in \eqref{vm}, we see that $\tilde v$ is orthogonal to the $n$ linearly independent vectors $u_1, \dots, u_n$ and is therefore identically zero. The condition $\tilde v=0$ is clearly also sufficient for $v$ to lie in the subspace orthogonal to $\nabla p_j(x), j=1,\dots,n$. Note that if $x_i$ is non-repeating for some $i$, then $0 = \tilde v_i= v_i\alpha_i$, from which we infer that $v_i = 0$. Now, for such a $v$, we have
\beqs\label{quadform}
 v^T(p_{n+1} - p_{n+1}^*)&\mathrm{diag}\big(\alpha_1q_n'(x_1),\dots, \alpha_Kq_n'(x_K)\big)v \\&= (p_{n+1} - p_{n+1}^*)\sum_{i=1}^nq_n'(x_i)\sum_{\substack{j\;\mathrm{s.t.}\\x_j = x_i}}\alpha_jv_j^2\\
 &=  (p_{n+1} - p_{n+1}^*)\sum_{\substack{1\leq i\leq n,\\x_i\;\mathrm{repeats}}}q_n'(x_i)\sum_{\substack{j\;\mathrm{s.t.}\\x_j = x_i}}\alpha_jv_j^2.
\eeqs
where in the last line we have used that $v_i=0$ for all $i$ such that $x_i$ is non-repeating. Now, let $i\in\{1,\dots, n\}$ be such that $x_i$ repeats. Note that as $v_j$ for $j$ such that $x_j=x_i$ range over the set satisfying $\sum_{\substack{j\;\mathrm{s.t.}\\x_j = x_i}}v_j\alpha_j = 0$, the number $\sum_{\substack{j\;\mathrm{s.t.}\\x_j = x_i}}\alpha_jv_j^2$ can take any value in $[0,\infty)$. We therefore see that the quadratic form of \eqref{quadform} is indeterminate if $q_n'(x_i)$ takes both positive and negative values for  repeating $x_i$. If $q_n'(x_i)$ are all of the same sign $s$ for repeating $x_i$, then the quadratic form is positive if $p_{n+1}(x) - p_{n+1}^*$ also has sign $s$, and negative if $p_{n+1}(x) - p_{n+1}^*$ has sign $-s$. Now, recall that $q_n$ is an order $n$ polynomial with roots $x_1, \dots, x_n$. Since these points are all distinct, this implies $q_n$ is a positive multiple of $(x-x_1)(x-x_2)\dots (x-x_n)$. The derivative of such a polynomial has alternating sign from root to root, and $q_n'(x_1) > 0$, since $x_1$ is the rightmost root on the real line. This finishes the proof of (b). 

Finally, (c) is shown in \cite{fan2020likelihood}. 
\end{proof}

\subsection{Algebraically Structured Models and Discussion}\label{subsec:algmodel} 
There are many important inference problems that are naturally modelled as GMMs with algebraic structure imposed on the centers. A motivating application is that of molecule imaging using Cryo-Electron Microscopy (cryo-EM), in which the goal is to reconstruct the density of a molecule given partial observations of it. The imaging data can be modeled by a GMM which generalizes the orbit recovery model in several ways. We describe the model in full generality, since it encapsulates most of the algebraically structured models of interest.  

The observations in cryo-EM are given by noisy projections of the molecule taken from different unknown viewing directions. Moreover, the molecule may be observed in one of several conformations. 
A common model assumption is to consider the noise to be Gaussian, in which case we can model the data by the GMM $Y = \sg Z + \bmta,$ where $Z$ is the additive Gaussian noise and $\bmta$ encodes the projection, rotation, and conformation of the molecule.

Specifically, let $\ta_j\in\R^d, j=1,\dots, K$ represent the densities of the molecule in its $K$ different conformations. These are the signals we wish to recover. Let $\bm\chi\in\{1,2,\dots, K\}$ be a random variable representing the probability to observe conformation $j=1,\dots, K$, with $\mathbb P(\chi = j) = \alpha_j, j=1,\dots, K$. This distribution is also unknown. Next, let $G$ be the group of rotations on $\R^d$, and ${\bf g}\sim \mathrm{Haar}(G)$ be a random variable which has uniform distribution over $G$. Finally, let $\Pi:\R^d\to \R^m$, with $m<d$, be the tomographic projection, a linear projection operator corresponding to the imaging procedure. We can then write the GMM as
\beq\label{cryoEM}Y = \sg Z + \bmta,\quad\bmta=\Pi({\bf g}\ta_{\bm\chi}).\eeq To summarize, the centers of this mixture, given by the support of $\bmta$, are the projections of the orbits under the continuous group $G$ of $K$ points in $\R^d$. The following are simplifications of this general model:
\begin{enumerate}
\item \textbf{Discrete Homogeneous Orbit Retrieval.} This is a type of orbit retrieval model~\eqref{orbitretrieval} described in Section~\ref{sec:main}. Here, there is no projection operator and only one orbit. One special case of interest is Multireference Alignment (MRA), in which the group $G= \{g_0,\dots, g_{d-1}\}$ is the group which acts on vectors in $\R^d$ by cyclically shifting their entries, i.e.
$$(g_j\ta)_k = \ta_{j+k\mod d},\quad j,k=0,\dots, d-1.$$ 
\item \textbf{Orbit Retrieval with non-uniform weights.} In this case, the distribution of ${\bf g}$ is not restricted to be uniform over $G$, and is unknown. 
\item \textbf{Continuous Orbit Retrieval.} Here, the group $G$ may be infinite. As an example, continuous MRA is a generalization of discrete MRA, in which shifts of entries are generalized to continuous shifts of periodic functions on the torus, i.e. $\tau_x\ta(y) = \ta(y+x\mod 1), x,y\in [0,1)$. The periodic functions are assumed bandlimited so that they can be represented in a finite-dimensional Fourier basis.  
\item \textbf{Heterogeneous Orbit Retrieval.} There is no projection operator, but the centers form $K>1$ orbits of a group $G$.
\end{enumerate}

A connection between the log-likelihood and the moments of the mixture was established in~\cite{bandeira17_opt, perry17_sample} for homogeneous MRA. In that paper, upper and lower bounds on the KL divergence (essentially the negative log likelihood) are given in the form of a series similar to ours, in which each term is the squared norm of the difference between true and estimated moments. This was then used to understand the sample complexity of the orbit retrieval problem, heavily exploiting the fact that the moments, due to the model's algebraic structure, correspond to invariant polynomials with respect to the group action. This showed that in the low SNR regime, the sample complexity of MRA increases from the standard $\mathcal O(1/\mathrm{SNR})$ to $\mathcal O(1/\mathrm{SNR}^3)$, a previously unexplained phenomenon first observed in experiments performed in the context of Cryo-EM~\cite{sig98}. This connection was then extended to the general setting~\eqref{cryoEM} in~\cite{bandeira17_moments}.

While these results help in understanding the statistical complexity of the models, they fail to explain why iterative methods such as EM appear to perform well in practice. (For more on EM and maximum likelihood estimation in cryo-EM, see e.g.~\cite{MLE-cryoEM}). The asymptotic expansion of the log likelihood and its connection to least squares moment matching is a step toward understanding why this is so. It is important to note however that understanding the roughness of the landscape of least squares of moments can be a highly non-trivial task. 

Nevertheless, even without a theoretical understanding of the roughness of the moments landscape, the connection between moment methods and more classical iterative approaches such as EM is itself of interest (and unexpected). The connection is especially tight if the method used for inverting the moments is minimization of an objective function of the form $\min_\ta\sum_k\lambda_k\|T_k(\ta) - T_k^*\|^2$. This is one of the methods used in a series of papers in which the moment-based approach was suggested for algebraically structured mixture models~\cite{bendory18_inv, Bendory_2019, ma2020mra2D, lan2020multitargetdetection}. In fact, numerical simulations in these papers demonstrate this connection. The experiments suggest that the two methods have similar performances for MRA and some of its extensions mentioned above. 

A particularly interesting example is that of Heterogeneous MRA, in which there are $M=Kd$ mixture components corresponding to the orbits under cyclic shifts of $K$ vectors in $\R^d$. Statistically, it is known~\cite{bandeira17_moments} that moments up to degree 3 are enough to resolve the model (provided the vectors are generic) even for $K$ growing linearly with $d$. However, numerical experiments, in which the vectors are chosen at random, suggest that algorithms start failing above $K\substack{>\\ \sim}\sqrt d$, 
and that the moment matching landscape has spurious local minima in this regime. It is conceivable that, for vectors chosen randomly from a Gaussian distribution, the third moment matching landscape is benign for $K\ll \sqrt{d}$ and riddled with spurious critical points when $K\substack{>\\ \sim}\sqrt d$.
 If such a phase transition is established, our expansion could then be used as a vehicle to transfer such results into an understanding of the performance of EM and similar methods.\\

\subsubsection{Random centers} Outside of algebraically structured GMMs, another interesting model is a GMM with ``average-case'' centers: take $M$ randomly sampled vectors in $\R^d$ from a Gaussian distribution and consider the GMM with these vectors as centers (and fixed isotropic covariances). The question of whether the mixture can be recovered from third moments is equivalent to low-rank tensor decomposition (the third moment tensor is a $d\times d\times d$ tensor with rank $\leq M$). This problem is believed to exhibit a statistical-to-computational gap: while the low rank decomposition is decidable for $M\ll d^2$ it is believed to be computationally hard for $M \gg d^{3/2}$~\cite{AlexThesis}. This is precisely the regime in which algorithms for moment inversion appear to fail in heterogeneous MRA, since the orbits of $K\sim\sqrt d$ vectors form the centers of a GMM with $M=Kd\sim d^{3/2}$ mixture components.  A characterization of the roughness of the landscape of low-rank tensor decomposition in these regimes could, with the help of our expansion, potentially be transferred to study the performance of EM in such a mixture model. \\

\subsection{Towards finite sample guarantees}
To make the connection rigorous between likelihood optimization and the series of minimization problems~\eqref{momentmatching} in low SNR models, one must prove that the path of gradient descent on the negative log likelihood is well-approximated by the stagewise least squares moment minimization. 
In order to study these algorithms in the finite sample case, one must also quantify the deviation of the sample log likelihood and its first two derivatives from the population log likelihood and its first two derivatives, respectively. 

\cite{fan2020likelihood} carries out this program to draw conclusions about log likelihood optimization in the case of homogeneous orbit retrieval. The tools developed in that paper lay the groundwork for analysis of more general models. In particular, the authors exploit the algebraic structure of the model to reparameterize the gradient descent dynamics in a basis of invariant polynomials under the group action. The varieties~\eqref{varieties} are then level sets of these polynomials, simplifying the analysis of the landscape of~\eqref{momentmatching}. 

We have not rigorously established the connection between the two landscapes for general GMMs or performed a finite sample analysis here, but we expect that doing so should be possible with the help of techniques developed in~\cite{fan2020likelihood}, as well as those used in the present paper for the derivation of the likelihood expansion.

\section{Log Likelihood Asymptotic Expansion}\label{sec:expansion} 

In this section, we prove the asymptotic expansion of the population log likelihood given in Lemma~\ref{mainthm}, highlighting key parts of the argument and deferring technical lemmas to the appendix. Recall that the log likelihood is given by 
\beqs
L(\rho;\,\rho_*) &=  \E_{\ta_*,Z}\log\,\E_{\ta}\lb\e\,\left(-\|\sg Z + \ta_*-\ta\|^2/2\sg^2\right)\rb,
\eeqs
where $\ta\sim\rho,\ta_*\sim\rho_*$, and $\rho,\rho_*$ are compactly supported distributions on $\R^d$. Note that in this section only, we use $\ta,\ta_*$ to denote random variables, rather than $\bmta,\bmta_*$. 
We begin with the following key observation. 
\begin{lemma}
Let $Z'\in\R^d$ be a random vector independent of $Z, \ta,$ and $\ta_*$ such that $Z'\sim\mathcal N(0, I)$. Then
\beq\label{simplification}
L(\rho;\,\rho_*) = -\frac d2 + \E_{\ta_*,Z}\log\,\E_{\ta, Z'}\lb \e\,\lp\frac{1}{\sg}(\ta-\ta_*)^T(Z+iZ')\rp\rb.
\eeq
where 
\end{lemma}
\begin{proof}
Consider the random variable $(\ta-\ta_*)^TZ'$. For $\ta,\ta_*$ fixed, it is a mean zero Gaussian with variance $\|\ta_* - \ta\|^2$ and hence has
characteristic function
\beq\label{Zprime}
\E_{Z'}\lb e^{it(\ta-\ta_*)^TZ'}\rb = \e\;\lp-\frac12t^2\|\ta-\ta_*\|^2\rp
\eeq
Now, we have
$$-\frac{1}{2\sg^2}\|\sg Z + \ta_*-\ta\|^2 = -\frac12\|Z\|^2 + \frac{1}{\sg}(\ta - \ta_*)^TZ - \frac{1}{2\sg^2}\|\ta_* - \ta\|^2.$$ Using \eqref{Zprime} with $t=1/\sg$, we then have
\beqs\E_{\ta}\big[\e\,&\left(-\frac{1}{2\sg^2}\|\sg Z + \ta_*-\ta\|^2\right)\big] \\&=\e\,\lp-\frac12\|Z\|^2\rp\E_{\ta,Z'}\lb\e\;\,\lp\frac{1}{\sg}(\ta - \ta_*)^TZ+\frac{1}{\sg}i(\ta - \ta_*)^TZ'\rp\rb\\
&=\e\,\lp-\frac12\|Z\|^2\rp\E_{\ta,Z'}\lb\e\;\,\lp\frac{1}{\sg}(\ta - \ta_*)^T(Z+iZ')\rp\rb.\eeqs
Taking the logarithm and expectation with respect to $\ta_*, Z$ gives \eqref{simplification}
\end{proof}

The remainder of the proof centers around a finite Taylor expansion about $t=0$ of 
\beq f(t, Z, \ta_*) =\log\,\E_{\ta, Z'}\lb \e\,\lp t(\ta-\ta_*)^T(Z+iZ')\rp\rb.\eeq Note that $f$ is $C^{\infty}$ in $t$ for every $Z$. Hence, for every $m$, $f$ has a finite Taylor expansion of the form
\beq\label{rep1}
f(t, Z,\ta_*) = \sum_{p=1}^{2m+1}\kappa_p\frac{t^p}{p!} + \partial_t^{2m+2}f(\xi)\frac{t^{2m+2}}{(2m+2)!}.
\eeq
Here, the $\kappa_j$ and $\xi$ both depend on $Z,\ta_*$, and $|\xi| < |t|.$ This expansion is valid for every $t\in\R$ and $Z, \ta_*\in\R^d$. Substituting \eqref{rep1} into \eqref{simplification} with $t=1/\sg$, we have
\beq\label{Lfiniteexpansion}
L(\rho;\,\rho_*) = -d/2 + \sum_{p=1}^{2m+1}\E_{\ta_*,Z}[\kappa_p]\frac{\sg^{-p}}{p!} + \E_{\ta_*,Z}\left[\partial_t^{2m+2}f(\xi)\right]\frac{\sg^{-2m-2}}{(2m+2)!}.
\eeq
To prove Lemma~\ref{mainthm}, it remains to compute the expectations of $\kappa_p,\,p=1,\dots, 2m+1$ and upper bound the expectation of the error term. The following theorem summarizes the results of these computations.
\begin{thm}\label{mainthm:revised} 
We have $$\E_{\ta_*,Z}[\kappa_{2k+1}]=0,\quad k=0,1,\dots$$ and
\beqs\frac{1}{(2k)!}E_{\ta_*,Z}[\kappa_{2k}]= -\frac{1}{2(k!)}\|T_k - T_k^*\|^2 &+ \mathbbm 1_{k>1}\lp\lla T_k,\; Q_k\rra + r_k\rp, \\&k=1,\dots, m,\eeqs where $r_k\in R_{2k}\lb T_{1:k-1}, T_{1:2k}^*\rb,$ $Q_k\in V_k\lb T_{1:k-1}, T_{1:k-1}^*\rb$, and $Q_k$ is such that $$Q_k(T_{1:k-1}^*, T_{1:k-1}^*) = 0.$$ The error term is bounded above by
$$\bigg|\E_{\ta_*,Z}\left[\partial_t^{2m+2}f(\xi)\right]\bigg|\frac{\sg^{-2m-2}}{(2m+2)!} \leq (m+1)!\lp C\frac{\delta}{\sg}\rp^{2m+2}\lp 1\vee \frac{\delta}{\sg}\rp^{2m+2},$$ where $C$ is a $d$-dependent constant and 
 $$\delta = \max\{\|x-x_*\|\mid x\in \mathrm{supp}(\ta), x_*\in\mathrm{supp}(\ta_*)\}.$$
\end{thm}

We now outline the main steps of the proof of Theorem~\ref{mainthm:revised}. In Section~\ref{sec:logsumexp}, we obtain expressions for the $\kappa_p$. We do so by taking advantage of generalized moment-cumulant relationships described below. We obtain
\beq\label{rep2}
f(t,Z,\ta_*) = \sum_{p=1}^{\infty}\frac{\kappa_p(Z, \ta_*)}{k!}t^k\quad\forall\,|t|<R_Z,
\eeq
where $R_Z$ is a $Z$-dependent radius of convergence, within which the series converges uniformly in $t$. That  the radius of convergence depends on $Z$ will not be an issue, as we only care about the coefficients $\kappa_p$ of \eqref{rep2}. Indeed, the $\kappa_p$ of \eqref{rep1} and \eqref{rep2} are the same. In Section~\ref{sec:errbound}, we upper bound the error term by explicitly computing $\partial_t^{n+1}f$ and bounding its $Z,\ta_*$-expectation.

In Section~\ref{sec:Zexpect}, we compute the $Z$-expectation of the $\kappa_p$, and in Section~\ref{sec:thetaexpect}, we compute the $\ta_*$-expectations of the $Z$-expectations. \\

In several key steps of the proof, we make use of the polynomials which express the cumulants of a distribution in terms of its moments. We will apply these moment-cumulant relations in a more general setting, in which the ``moments'' are coefficients of any Taylor expansion satisfying certain conditions. 
Before proceeding with the proof, we describe these generalized moment-cumulant relations.

\subsection{Generalized Moment-Cumulant Relations}
Let $X$ be a random variable with moment-generating function 
$$M_X(t) = \E\lb e^{tX}\rb = 1 + \sum_{k=1}^{\infty}\frac{\mu_k}{k!}t^k,\quad |t|<R,$$ and cumulant generating function
$$\kappa_X(t) = \log M_X(t) = \sum_{m=1}^{\infty}\frac{\kappa_m}{m!}t^m.$$ The $\mu_k$ are the moments of $X$, $\mu_k = \E[X^k]$, and the cumulants $\kappa_m$ are given by the following polynomials $\kappa_m(\mu_1,\dots, \mu_m)$:
\beq\label{kappa-mu}
\kappa_m(\mu_1,\dots, \mu_m) = \sum_{\lambda\in S_m}c_{\lambda}\prod_{k\in\lambda}\mu_k,
\eeq
where $S_m$ is the set of all finite lists $\lambda$ of positive integers whose sum is $m$. The $c_{\lambda}$ are universal constants, and we will rarely need to know their exact values. As an example,
$$\kappa_4 = c_4\mu_4 + c_{3,1}\mu_3\mu_1 + c_{2,2}\mu_2^2 + c_{2,1,1}\mu_2\mu_1^2.$$
The moment-cumulant relations \eqref{kappa-mu} are typically applied in the context of random variables. However, they arise in a more general context: for a function $f$ with Taylor series coefficients $\mu_k$, \eqref{kappa-mu} describes how the Taylor series coefficients $\kappa_m$ of $\log f$ relate to the $\mu_k$. More concretely, we have the following result, proved in Proposition~\ref{app:prop:logM} of the Appendix. 
\begin{prop}\label{cum-mom-prop}
Let $M(t)$ be a real analytic function in the neighborhood $|t-t_0| < R$, for which $M(t_0)\neq 0$. If $$\sup_{|t-t_0|<R}\left|\frac{M(t)}{M(0)} - 1\right| < 1,$$ then 
\beq\label{gen-kappa-mu}\frac{d^m}{dt^m}\log M(t)\big\vert_{t=t_0} = \kappa_m\lp \frac{M^{(1)}(t_0)}{M(t_0)},\frac{M^{(2)}(t_0)}{M(t_0)}\dots, \frac{M^{(m)}(t_0)}{M(t_0)}\rp,\eeq where the function $\kappa_m$ is defined by \eqref{kappa-mu}.

In particular, if $u(t)$ is given by the convergent series expansion 
$$u(t) = \sum_{k=1}^{\infty}\frac{\mu_k}{k!}(t-t_0)^k,\quad |t-t_0| < R,$$ and $$\sup_{|t-t_0|<R}|u(t)|< 1,$$ then
$$\log\lp 1+u(t) \rp = \sum_{m=1}^{\infty}\frac{\kappa_m}{m!}(t-t_0)^m, \quad |t-t_0| < R,$$ where the $\kappa_m$ are given by \eqref{kappa-mu}.
\end{prop}
\begin{remark}
Note that \eqref{gen-kappa-mu} holds for any function $M(t)$ with $m$ derivatives in a neighborhood of $t_0$, and $M(t_0)\neq 0$. In other words, \eqref{gen-kappa-mu} is also simply the expression for the $m$th derivative of $\log M$ that comes from repeatedly applying the rules of differentiation. 
\end{remark}

In the following four sections of the proof, we will make frequent use of the identity
$$\la x^{\otimes k},\; y^{\otimes k}\ra = (x^Ty)^k, \quad x,y\in\R^d, \;k=1,2,\dots.$$ 
\subsection{Log-Sum-Exp Taylor Series}\label{sec:logsumexp}
Recall that 
\beq f(t, Z, \ta_*) =\log\,\E_{\ta}\E_{Z'}\lb \e\,\lp t(\ta-\ta_*)^T(Z+iZ')\rp\rb.\eeq
In this section, we derive the series expansion
$$f(t,Z,\ta_*) = \sum_{p=1}^{\infty}\frac{\kappa_p(Z, \ta_*)}{p!}t^p\quad\forall\,|t|<R_Z,$$ for $R_Z$ to be specified. Throughout the section, $Z$ and $\ta_*$ are considered constant. Let $W = Z+ iZ'$, $w = \ta-\ta_*$, and $$M(t) = M(t,Z,\ta_*) = \E_{\ta,Z'}\lb \e\,\lp tw^TW\rp\rb$$ so that $f = \log M$. 
Recall that $\delta =\sup\{\|x-x_*\|\mid x\in\mathrm{supp}(\ta), x_*\in\mathrm{supp}(\ta_*)\}$, so that $\|w\|\leq\delta$. Now, 
$$ \e\,\lp tw^TW\rp = \lim_{m\to\infty}\left(1+\sum_{k=1}^m(w^TW)^k\frac{t^k}{k!}\right).$$ The partial sums are each bounded in absolute value by $\e\,\lp\delta t(\|Z\|+\|Z'\|)\rp$ which has finite $\ta,Z'$-expectation. We can therefore interchange summation and expectation to get
\beqs\label{expandpj}
M(t)&=\E_{\ta,Z'}\lb \e\,\lp tw^TW\rp\rb\\ &=1+\sum_{k=1}^{\infty}\E_{\ta,Z'}\left[(w^TW)^k\right]\frac{t^k}{k!}\\
&=  1 + \sum_{k=1}^{\infty}\bigg\la \E_{\ta}\lb w^{\otimes k}\rb,\; \E_{Z'}\left[W^{\otimes k}\right]\bigg\ra\frac{t^k}{k!}\quad\forall\;t\in\R.
\eeqs
Denote the coefficients in this expansion by $$\mu_k = \lla \E_{\ta}\lb w^{\otimes k}\rb,\; \E_{Z'}\left[W^{\otimes k}\right]\rra = \lla \E_{\ta}\lb w^{\otimes k}\rb,\; \E_{\mathrm{Im}\,W}\left[W^{\otimes k}\right]\rra.$$ We now apply the generalized moment-cumulant relations to Taylor expand the logarithm of $M(t)$. In order to do so, we first limit the range of $t$ to ensure $M(t)$ remains in a small neighborhood of $1$. Now, we can write $M(t)$ as
$$M(t)= \E_{\ta}\;\e\,\lp w^TZt -\frac12t^2\|w\|^2\rp.$$ Using that $\|w\|\leq\delta$, we have $$\left|w^TZt -\frac12t^2\|w\|^2\right| \leq \frac12\quad\mathrm{if}\;|t|<R_Z= \frac{1}{\delta\max(4\|Z\|, \sqrt2)}.$$
Therefore, for $|t|<R_Z$ we have
\beqs
\left|M(t) - 1\right|&\leq \E_{\ta}\left|\e\,\lp w^TZt -\frac12t^2\|w\|^2\rp - 1\right|\\
&\leq \frac74 \E_{\ta}\left|w^TZt -\frac12t^2\|w\|^2\right|\leq \frac78,
\eeqs
where we have used the fact that $|e^x-1| < 7|x|/4$ when $|x|<1$. For $|t|<R_Z$ we can thus make use of the generalized moment-cumulant relations of Proposition~\ref{cum-mom-prop} to write
$$f(t) = \log\;M(t) =\sum_{p=1}^{\infty}\kappa_p\frac{t^p}{p!},$$ where
\beqs
\kappa_p &=\kappa_p(Z,\ta_*) = \sum_{\lambda\in S_p}c_{\lambda}\prod_{\ell\in\lambda}\mu_{\ell} \\
&= \sum_{\lambda\in S_p}c_{\lambda}\prod_{\ell\in\lambda}\lla \E_{\ta}\lb w^{\otimes \ell}\rb,\; \E_{\mathrm{Im}\,W}\left[W^{\otimes \ell}\right]\rra.\\
&=\sum_{\lambda\in S_p}c_{\lambda}\lla\bigotimes_{\ell\in\lambda}\E_{\ta}\lb w^{\otimes \ell}\rb,\;\E_{\mathrm{Im}\,W_{\lambda}}\left[\bigotimes_{\ell\in\lambda}W_\ell^{\otimes \ell}\right]\rra. 
\eeqs
In the third line, $W_{\lambda}$ denotes the set $\{W_\ell\mid\ell\in\lambda\}$. We replaced $W=Z+iZ'$ with vectors $W_\ell = Z + iZ_\ell, \,\ell\in\lambda,$ where the $Z_\ell\in\R^d$ are i.i.d. standard normal and independent of $Z$. This allowed us to write the product of expectations as the expectation of a product.

\subsection{Error Term Upper Bound}\label{sec:errbound} In this section, the constant $C$ depends on $d$ only and may change value from line to line. Recall that the error term is given by $$\E_{Z, \ta_*}\lb\partial_t^{2m+2}f(\xi)\rb\frac{\sg^{-2m-2}}{(2m+2)!},$$ for $\xi$ such that $0<\xi<1/\sg$. To bound it, we first compute $\partial_t^{n}f$ at a generic point $t$ (with $n=2m+2$). Recall that $w=\ta-\ta_*$, $W=Z+iZ'$, and $f=\log M$ for $M(t, Z,\ta_*) = \E_{\ta,Z'}\lb \e\,\lp tW^Tw\rp\rb$. By \eqref{gen-kappa-mu}, we have
\beqs\label{df}
\partial_t^nf &= \kappa_n\left(\partial_t^1M/M,\dots, \partial_t^nM/M\right)\\
&= \sum_{\lambda\in S_n}c_{\lambda}\prod_{\ell\in\lambda}\frac{\partial_t^\ell M}{M}\eeqs at any point $t$, since $M(t)$ is never zero. We therefore have the error bound
\beq\label{expabsdf}
\frac{\sg^{-n}}{n!}\E_{Z,\ta_*}\left|\partial_t^nf\right| \leq  \frac{\sg^{-n}}{n!}\E_{\ta_*}\sum_{\lambda\in S_n}|c_{\lambda}|\,\E_Z\prod_{\ell\in\lambda}\left|\frac{\partial_t^\ell M}{M}\right|.\eeq

We now compute the $t$-derivative of $M$ in the following indirect way, which will yield an expression that is simpler to bound. Let $t=t_0 + s$; we will write $M$ in terms of $s$ and take its derivative at $s=0$. We have $$M(t_0+s) = \E_{\ta}\lb e^{(t_0+s)w^TZ}\E_{Z'}\;e^{i(t_0+s)w^TZ'}\rb,$$ and note that 
$$\E_{Z'}\lb e^{i(t_0+s)w^TZ'}\rb = e^{-t_0s\|w\|^2}\E_{Z'}\lb e^{it_0w^TZ'}\rb\E_{Z'}\lb e^{isw^TZ'}\rb.$$ Therefore, 
\beqs 
M(t_0+s) &= \E_{\ta}\lb e^{-t_0s\|w\|^2}\E_{Z'}\lb e^{t_0w^TW}\rb\E_{Z'}\lb e^{sw^TW}\rb\rb\\
&= \E_{\ta}\lb \E_{Z'}\lb e^{t_0w^TW}\rb\E_{Z'}\lb e^{sw^T(W-t_0w)}\rb\rb.\eeqs We now take the derivative at $s=0$, passing it inside both expectations. This is justified since the resulting derivative is absolutely integrable. We obtain
\beqs\partial_t^\ell M(t_0) = \E_{\ta}\lb \E_{Z'}\lb e^{t_0w^TW}\rb\E_{Z'}\lb \lp w^T(W-t_0w)\rp^\ell\rb\rb.\eeqs
Noting that $\E_{Z'}\lb e^{tw^TW}\rb >0,$ we have
\beqs\label{M-bound}
\left|\frac{\partial_t^\ell M(t)}{M(t)}\right| &\leq \frac{\E_{\ta}\lb \E_{Z'}\, e^{tw^TW}\E_{Z'}\left| w^T(W-tw)\right|^\ell\rb}{\E_{\ta}\lb \E_{Z'}\, e^{tw^TW}\rb}\\
&\leq\sup_{\ta}\E_{Z'}\left| w^T(W-tw)\right|^\ell \\
&\leq \delta^\ell\E_{Z'}(\|W\| + |t|\delta)^\ell
\eeqs
Now, fix $\lambda\in S_n$, and let $W_\ell = Z+iZ'_\ell,\,\ell\in\lambda$, where $Z'_\ell$ are independent copies of $Z'$. We have
\beqs
\E_Z\left|\prod_{\ell\in\lambda}\frac{\partial_t^\ell M}{M}\right| &\leq\E_Z\prod_{\ell\in\lambda}\delta^\ell\E_{\mathrm{Im}\,W_\ell}(\|W_\ell\| + |t|\delta)^\ell = \delta^n\E\prod_{\ell\in\lambda}(\|W_\ell\| + |t|\delta)^\ell\\
&\leq (2\delta)^n\E\lp \frac12(|t|\delta)^n + \frac12\sum_{\ell\in\lambda}\frac{\ell}{n}\|W_\ell\|^n\rp = (2\delta)^n\lp\frac12(|t|\delta)^n + \frac12\E\|W\|^n \rp\\
&\leq \lp\frac n2\rp!(C\delta)^n(1\vee |t|\delta)^n
\eeqs
where expectation with no subscripts denotes expectation with respect to all random variables and $C$ is a constant depending on $d$ only. Now, note that this upper bound is independent of $\lambda$.  We also have
$$\sum_{\lambda\in S_n}|c_\lambda| \leq n^n\leq n!e^n.$$ This follows from ~\cite{fan2020likelihood}, in which it was shown that the sum of the absolute values of the coefficients arising in an $n$th order multivariate cumulant is bounded above by $n^n$. It is straightforward to see that this bound applies to univariate cumulants as well. Substituting these bounds in the error bound \eqref{expabsdf} and using $|t|=|\xi|<1/\sg$, we obtain
\beqs\frac{\sg^{-n}}{n!}\E_{Z, \ta_*}\left|\partial_t^nf\right|&\leq \frac{\sg^{-n}}{n!}\lp\frac n2\rp!(C\delta)^n(1\vee |t|\delta)^n\sum_{\lambda\in S_n}|c_\lambda| \\
&\leq \lp\frac n2\rp!\lp C\frac{\delta}{\sg}\rp^n\lp1\vee \frac{\delta}{\sg}\rp^n.
\eeqs
Taking $n=2m+2$ gives the desired error bound. 

\subsection{Z-Expectation of Cumulants}\label{sec:Zexpect} We will use the following notation in this section: for a list $\lambda$, we define $\ell_{\max}$ as the maximum element in the list, and $\lambda\setminus\ell$ is the list with one copy of $\ell$ removed. Now, recall that $w=\ta-\ta_*$, and 
$$\kappa_p = \sum_{\lambda\in S_p}c_{\lambda}\lla\bigotimes_{\ell\in\lambda}\E_{\ta}\lb w^{\otimes \ell}\rb,\;\E_{\mathrm{Im}\,W_{\lambda}}\left[\bigotimes_{\ell\in\lambda}W_\ell^{\otimes \ell}\right]\rra,$$ where $W_{\lambda}$ denotes the set $\{W_\ell\mid\ell\in\lambda\}$ and $W_\ell=Z+iZ_\ell$, where $Z, Z_\ell,\,\ell\in\lambda$ are independent standard normal vectors in $\R^d$. 
We therefore have
$$\E_{\ta_*,Z}\lb\kappa_p\rb = \sum_{\lambda\in S_p}c_{\lambda}\lla\E_{\ta_*}\lb\bigotimes_{\ell\in\lambda}\E_{\ta}\lb w^{\otimes \ell}\rb\rb,\;\E_{W_{\lambda}}\left[\bigotimes_{\ell\in\lambda}W_\ell^{\otimes \ell}\right]\rra.$$ We make two initial observations:
\begin{enumerate}
\item If $p=2k+1$ is odd, then the expectation on the right side of the inner product is zero, so that $\E[\kappa_{2k+1}]=0$. This follows from the fact that all Gaussian random variables appearing in the expectation are mean zero, and the total order of the product is $2k+1$. 
\item Suppose $p=2k$. If $\lambda$ is such that all $\ell\in\lambda$ are less than $k$, then the $\ta_*$ expectation on the left is a polynomial only of moment tensors $T_p$ with $p<k$ (though it could depend on moment tensors $T_p^*$ for $p$ up to $2k$.) This is proved in Lemma~\ref{lemma:thstar-expect} in the appendix. Therefore, the inner product for such $\lambda$s will contribute only to the remainder polynomial $r_k(T_{1:k-1}, T_{1:2k}^*).$
\end{enumerate}
We therefore consider $p=2k$, and discard from the sum those $\lambda\in S_{2k}$ in which all numbers are less than $k$. Let $\ell_{\max}$ denote the maximum number in a list $\lambda$. We have 
\beqs\label{EkappaZ}\E_{\ta_*,Z}\lb\kappa_{2k}\rb = \E_{\ta_*}\sum_{\substack{\lambda\in S_{2k}\\\ell_{\max}\geq k}}c_{\lambda}&\lla\bigotimes_{\ell\in\lambda}\E_{\ta}\lb w^{\otimes \ell}\rb,\;\E\left[\bigotimes_{\ell\in\lambda}(Z+iZ_\ell)^{\otimes \ell}\right]\rra\\& + r_k(T_{1:k-1}, T_{1:2k}^*).\eeqs

Using results in \cite{moments}, we now prove the following proposition.
\begin{prop}\label{EZW}
Let $\lambda\in S_{2k}$ be such that $\ell_{\max}$, the maximum of the list, is at least $k$. Let $Z, Z_\ell,\ell\in\lambda$ be i.i.d. standard normal vectors in $\R^d$. Then
\beq
\E\left[\bigotimes_{\ell\in\lambda}(Z+iZ_\ell)^{\otimes \ell}\right] = \begin{cases}0,\quad &\ell_{\max}>k\\
2^{-k}\E\left[W^{\otimes k}\otimes \overline W^{\otimes k}\right],\quad &\ell_{\max}=k,\end{cases}
\eeq
where $W = Z+iZ_{\ell_{\max}}$. 
\end{prop}
The proof relies on the following result of \cite{moments} 
\begin{prop}\label{momentprop}[Proposition 2 of \cite{moments}]
Let $V = \lp V^1,\dots, V^p\rp\in\mathbb C^p$ be jointly circularly symmetric. Then $$\E\left[\lp V^1\rp^{k_1}\dots \lp V^p\rp^{k_p}\overline {V^1}^{m_1}\dots\overline {V^p}^{m_p}\right]\neq 0$$ only if $k_1 + \dots + k_p = m_1 + \dots + m_p$. 
\end{prop}
\begin{proof}[Proof of Proposition~\ref{EZW}]
Note that each $Z+iZ_\ell$ is a circularly symmetric random vector. However, $(Z+iZ_\ell)_{\ell\in\lambda}\in\mathbb C^{|\lambda|d}$ is not jointly circularly symmetric, because the real parts are correlated while the imaginary parts are independent. We can nevertheless take advantage of Proposition~\ref{momentprop} as follows:  first, define $Z'_\ell\in\R^d,\,\ell\in\lambda\setminus\ell_{\max}$ to be standard normal and independent of $Z, Z_\ell,\,\ell\in\lambda$ and of each other. Define $W = Z+ iZ_{\ell_{\max}}$, and $W_\ell= Z'_\ell+iZ_\ell,\,\ell\in\lambda\setminus\ell_{\max}$. Thus, $(W;\,(W_\ell)_{\ell\in\lambda\setminus\lambda_{\max}})\in\mathbb C^{|\lambda|d}$ is circularly symmetric. Now, for $\ell\in\lambda\setminus\ell_{\max}$ we write $$Z+iZ_\ell = \frac12(W+\overline W) + \frac12(W_\ell - \overline W_\ell) = \frac12(W+W_\ell + \overline{W}-\overline{W_\ell}).$$ With this notation, we have
 $$\bigotimes_{\ell\in\lambda}(Z+iZ_\ell)^{\otimes \ell} = W^{\otimes \ell_{\max}}\otimes\lp\frac12\rp^{2k-\ell_{\max}}\bigotimes_{\ell\in\lambda\setminus\ell_{\max}}(W+W_\ell +\overline W - \overline W_\ell)^{\otimes \ell}.$$ Note that each entry in this tensor is of the form given in the proposition, i.e. it is a product of some number of conjugated and unconjugated complex random variables which are jointly circularly symmetric Gaussian. We count how many conjugated and unconjugated variables appear in a typical entry of this tensor. There are at least $\ell_{\max}$ unconjugated variables (from the $\ell_{\max}$ copies of $W$) , and at most $2k - \ell_{\max} \leq \ell_{\max}$ conjugated variables. Thus, we immediately see that the expectation is zero if $\ell_{\max}  > k$. If $\ell_{\max}=k$, the expectations of all terms with fewer than $k$ conjugated variables are zero. Writing $k$ instead of $\ell_{\max}$, we then have
$$\E\left[W^{\otimes k}\otimes\bigotimes_{\ell\in\lambda\setminus k}(W+W_\ell +\overline W - \overline W_\ell)^{\otimes \ell}\right]=\E\left[W^{\otimes k}\otimes\bigotimes_{\ell\in\lambda\setminus k}(\overline W- \overline W_\ell)^{\otimes \ell}\right].$$ But note that $W_\ell$ is independent of $W$, so the expectation of products involving both entries of $W$ and entries of $\overline W_\ell$ will split up into a product of two expectations, one of which involves only entries of $\overline W_\ell$. But this expectation is zero by the proposition, since it involves only conjugated variables. Hence only the products involving $W$ alone survive. We obtain
$$\E\left[W^{\otimes k}\otimes\bigotimes_{\ell\in\lambda\setminus k}(\overline W - \overline W_\ell)^{\otimes \ell}\right] = \E\left[W^{\otimes k}\otimes\bigotimes_{\ell\in\lambda\setminus k}\overline W^{\otimes \ell}\right] = \E\left[W^{\otimes k}\otimes \overline W^{\otimes k}\right].$$ We recall the additional factor $\lp\frac12\rp^{2k-\ell_{\max}}=2^{-k}$ to conclude.
\end{proof}
Substituting this back into the cumulant formula, we have 
\beqs\E_Z\lb\kappa_{2k}\rb&=2^{-k}\sum_{\substack{\lambda\in S_{2k}\\\ell_{\max}=k}}c_{\lambda}\lla \E_{\ta}\lb w^{\otimes k}\rb\otimes\bigotimes_{\ell\in\lambda\setminus k}\E_{\ta}\lb w^{\otimes \ell}\rb,\;\E\left[W^{\otimes k}\otimes \overline W^{\otimes k}\right]\rra\\
&= 2^{-k}\sum_{\substack{\lambda\in S_{2k}\\\ell_{\max}=k}}c_{\lambda}\E_W\lb\E_{\ta}\lb(w^TW)^k\rb\prod_{\ell\in\lambda\setminus k}\E_{\ta}\lb(w^T\overline W)^{\ell}\rb\rb.\eeqs
This expectation is straightforward to evaluate (see Proposition~\ref{app:ZexpectFinal} in the Appendix), and we obtain
\beqs\label{EkappaAgain}
\E_Z\left[\kappa_{2k}\right] &= k!\sum_{\substack{\lambda\in S_{2k}\\\ell_{\max}=k}}c_{\lambda}\lla \E_{\ta}\lb w^{\otimes k}\rb, \; \bigotimes_{\ell\in\lambda\setminus k}\E_{\ta}\lb w^{\otimes \ell}\rb\rra \\
&= k!\lla \E_{\ta}\lb w^{\otimes k}\rb, \; \sum_{\substack{\lambda\in S_{2k}\\\ell_{\max}=k}}c_{\lambda}\bigotimes_{\ell\in\lambda\setminus k}\E_{\ta}\lb w^{\otimes \ell}\rb\rra.
\eeqs
\subsection{$\ta_*$-Expectation of Cumulants}\label{sec:thetaexpect}


Recall that $w = \ta - \ta_*$, and that
\beq
T_\ell = \E\lb\ta^{\otimes \ell}\rb,  \quad T_\ell^* = \E\lb\ta_*^{\otimes \ell}\rb, \quad \ell=1,2,\dots.
\eeq
We also remind the reader of the definitions of total moment order and the spaces $V_k, R_k$:
\begin{defn*} Consider $$S = \bigotimes_{i=1}^n S_{k_i}=S_{k_1}\otimes S_{k_2}\otimes\dots\otimes S_{k_n},$$ where $S_{k_i}$ is either $T_{k_i}$ or $T_{k_i}^*$. We define the \textbf{total moment order} of $S$ to be $\sum_{i=1}^n k_i,$ i.e. the sum of all moment orders. We also say that the total moment order of each entry of $S$ is $\sum_{i=1}^n k_i$; in other words, the total moment order of products of entries of moment tensors is the sum of all moment orders in the product. 
\end{defn*}
\begin{defn*}
We define $V_k\lb T_{1:m}, T_{1:n}^*\rb$ as the set of all constant coefficient linear combinations of outer products of moment tensors $T_j, j\leq m,\, T_\ell^*,\ell\leq n$, of total moment order $k$. We define $R_k\lb T_{1:m}, T_{1:n}^*\rb$ as the set of all constant coefficient linear combinations of products of entries of moment tensors $T_j,T_\ell^*,j\leq m,\ell\leq n$, of total moment order $k$.
\end{defn*}

 In the previous section, we have shown that 
\beq\label{EZkappa}\E_Z\left[\kappa_{2k}\right] = k!\lla \E_{\ta}\lb w^{\otimes k}\rb, \; \sum_{\substack{\lambda\in S_{2k}\\\ell_{\max}=k}}c_{\lambda}\bigotimes_{\ell\in\lambda\setminus k}\E_{\ta}\lb w^{\otimes \ell}\rb\rra + s(T_{1:k-1},\ta_*),\eeq where $s$ is such that $\E_{\ta_*}\lb s(T_{1:k-1},\ta_*)\rb \in R_k\lb T_{1:k-1}, T_{1:2k}^*\rb$.  For brevity, define the tensor on the right of \eqref{EZkappa} by $J_k$, that is
$$J_k = \sum_{\substack{\lambda\in S_{2k}\\\ell_{\max}=k}}c_{\lambda}\bigotimes_{\ell\in\lambda\setminus k}\E_{\ta}\lb w^{\otimes \ell}\rb.$$ To get a sense of $J_k$, let us write out $J_3$:
$$J_3=c_{3,3}\E_{\ta}\lb w^{\otimes 3}\rb + c_{3,2,1}\E_{\ta}\lb w^{\otimes 2}\rb\otimes \E_{\ta}\lb w\rb + c_{1,1,1}\E_{\ta}\lb w\rb\otimes \E_{\ta}\lb w\rb\otimes \E_{\ta}\lb w\rb$$

In this section, we take the $\ta_*$-expectation of the inner product in \eqref{EZkappa}, that is, of $\la \E_{\ta}\lb w^{\otimes k}\rb, J_k\ra$.  We are interested only in terms involving $T_k$, the highest order moment tensor of $\ta$ appearing in this inner product.  Therefore, we can separate out $T_k$ on each side of the inner product and discard terms in which $T_k$ appears on neither side. The discarded terms will collectively be denoted $s(T_{1:k-1},\ta_*)$. Note that the $\ta$-moment tensors which arise upon expanding $\E_{\ta}\lb(\ta-\ta_*)^{\otimes p}\rb$ are $T_j,\,j\leq p$. Therefore, $T_k$ appears only through $\E_{\ta}\lb(\ta-\ta_*)^{\otimes k}\rb$, which arises on the left of the inner product with coefficient $1$ and on the right with coefficient $c_{k,k}$. We have
\beqs
\bigg\la\E_{\ta}\lb w^{\otimes k}\rb, \; J_k\bigg\ra &= \bigg\la \E_{\ta}\lb \ta^{\otimes k}\rb, \; J_k\bigg\ra
+ \bigg\la \E_{\ta}\lb w^{\otimes k}-\ta^{\otimes k}\rb, \; c_{k,k}\E_{\ta}\lb \ta^{\otimes k}\rb\bigg\ra + s(T_{1:k-1},\ta_*)\\
&= \bigg\la \E_{\ta}\lb \ta^{\otimes k}\rb, \; J_k + c_{k,k}\E_{\ta}\lb w^{\otimes k}-\ta^{\otimes k}\rb\bigg\ra+ s(T_{1:k-1},\ta_*).
\eeqs
We may then write
\beqs\label{toFinish}
\E_{\ta_*}\E_Z\lb\kappa_{2k}\rb &= k!\E_{\ta_*}\bigg\la\E_{\ta}\lb w^{\otimes k}\rb, \; J_k\bigg\ra \\
&= k!\E_{\ta_*} \bigg\la \E_{\ta}\lb \ta^{\otimes k}\rb, \; J_k + c_{k,k}\E_{\ta}\lb w^{\otimes k}-\ta^{\otimes k}\rb\bigg\ra +  \E_{\ta_*}\lb s(T_{1:k-1},\ta_*) \rb \\
&= k!\E_{\ta'}\E_{\ta_*}\bigg\la \ta'^{\otimes k}, \; J_k + c_{k,k}\E_{\ta}\lb w^{\otimes k}-\ta^{\otimes k}\rb\bigg\ra +r_k(T_{1:k-1}, T_{1:2k}^*).
\eeqs
To get the last line, we substitute $\E_{\ta}\lb \ta^{\otimes k}\rb$ with $\E_{\ta'}\lb \ta'^{\otimes k}\rb$, where $\ta'$ is an independent copy of $\ta$. We then pull the $\ta'$-expectation out of the inner product. The fact that $\E_{\ta_*}\lb s(T_{1:k-1},\ta_*)\rb = r_k(T_{1:k-1}, T_{1:2k}^*)$ is proved in the appendix, see Lemma~\ref{r-nu}. 

The following lemma will complete the proof of Theorem~\ref{mainthm:revised} .
\begin{lemma}\label{lma:toFinish} Let $x\in\R^d$ be constant. Then
\beq\label{lma:eqn:toFinish}\E_{\ta_*}\bigg\la x^{\otimes k}, \; J_k + c_{k,k}\E_{\ta}\lb w^{\otimes k}-\ta^{\otimes k}\rb\bigg\ra = \bigg\la x^{\otimes k},\; -2c_{k,k}T_k^* + c_{k,k}T_k + Q_k\bigg\ra,\eeq where $Q_k\in V_k\lb T_{1:k-1}, T_{1:k-1}^*\rb$. Moreover, $Q$ is independent of $x$ and satisfies $Q_{k}(T_{1:k-1}^*, T_{1:k-1}^*)=0$.
\end{lemma}
To finish the proof of Theorem~\ref{mainthm:revised}, note that substituting $x=\ta'$ and taking the $\ta'$-expectation gives
$$\E_{\ta'}\E_{\ta_*}\bigg\la \ta'^{\otimes k}, \; J_k + c_{k,k}\E_{\ta}\lb w^{\otimes k}-\ta^{\otimes k}\rb\bigg\ra = c_{k,k}\|T_k\|^2 -2c_{k,k}\la T_k,\, T_k^*\ra+ \la T_k, Q_k\ra.$$
Substituting into \eqref{toFinish} and dividing by $(2k)!$ gives
$$\frac{1}{(2k)!}\E_{\ta_*}\E_Z\lb\kappa_{2k}\rb = \frac{k!}{(2k)!}c_{k,k}\|T_k-T_k^*\|^2 + \la T_k, Q_k\ra + r_k,$$ where we have absorbed $c_{k,k}\|T_k^*\|^2$ into $r_k$.  In Lemma~\ref{cMM} in the appendix, we show that $$c_{k,k} = -\frac12{2k\choose k},\;\mathrm{so}\,\mathrm{that}\;\frac{k!}{(2k)!}c_{k,k} = -\frac{1}{2(k!)}.$$ Hence, the coefficient in front of the squared difference of $k$th moments is as in Theorem~\ref{mainthm:revised}.
Now, Lemma~\ref{lma:toFinish} will follow from the following Lemma:
\begin{lemma}\label{lma:toFinishForReal}
Let $X, X_*\in\R$ be bounded random variables, and let $W=X-X_*$. Let
$$\mu_j = \E\lb X^j\rb, \quad \mu_j^* = \E\lb X_*^j\rb, \quad \omega_j = \E_X\lb W^j\rb.$$ (Note that $\omega_j$ is random with respect to $X_*$.) Let $c_\lambda$ be the universal coefficients of the moment-cumulant relations ~\eqref{kappa-mu}. Then
\beq\label{XXstarW} \E_{X_*}\bigg[c_{k,k}\omega_k + \bigg(\sum_{\substack{\lambda\in S_{2k}\\\ell_{\max}=k}}c_{\lambda}\prod_{\ell\in\lambda\setminus k}\omega_\ell\bigg)\bigg] = 2c_{k,k}(\mu_k - \mu_k^*) + \pi(\mu_{1:k-1}, \mu_{1:k-1}^*),\eeq
where $\pi\in R[\mu_{1:k-1}, \mu_{1:k-1}^*]$ and satisfies $\pi(\mu_{1:k-1}^*, \mu_{1:k-1}^*) = 0$. 
\end{lemma}
We show how Lemma~\ref{lma:toFinish} follows from Lemma~\ref{lma:toFinishForReal} in the appendix; see Lemma~\ref{app:equiv}. The idea is to take $X = x^T\ta$ and $X_* = x^T\ta_*$.
\begin{proof}[Proof of Lemma~\ref{lma:toFinishForReal}]
Define the polynomials
 $$q(t) =  \sum_{j=1}^k\mu_jt^j/j!,\quad q_*(t) =  \sum_{j=1}^k\mu^*_jt^j/j!, \quad q_\omega(t) =  \sum_{j=1}^k\omega_jt^j/j!.$$ The proof consists of the following sequence of computations. 
%
\begin{enumerate}[wide=0pt,label=\textbf{\Alph*.}]
\item We show that
$$\omega_k\bigg (c_{k,k}\omega_k + \sum_{\substack{\lambda\in S_{2k}\\\ell_{\max}=k}}c_{\lambda}\prod_{\ell\in\lambda\setminus p}\omega_\ell\bigg) =c_{k,k}\omega_k^2  -\kappa_{2k}(M^{(1)}(0), \dots, M^{(2k)}(0)),$$ where
$$M(t) = 1-\frac{\omega_kt^k}{k!}(1+q_\omega(t))^{-1}.$$
\item We show that
$$c_{k,k}\omega_k^2 - \kappa_{2k}(M^{(1)}(0), \dots, M^{(2k)}(0))= -2c_{k,k}\omega_k\frac{d^k}{dt^k}(1+q_\omega(t))^{-1}\big\vert_{t=0},$$ and hence
\beq\label{ckomega} c_{k,k}\omega_k + \sum_{\substack{\lambda\in S_{2k}\\\ell_{\max}=k}}c_{\lambda}\prod_{\ell\in\lambda\setminus p}\omega_\ell = -2c_{k,k}\frac{d^k}{dt^k}(1+q_\omega(t))^{-1}\big\vert_{t=0}.\eeq
This is true even if $\omega_k=0$. Indeed, \eqref{ckomega} is true if $\omega_k\neq 0$, and both sides of the equation are continuous with respect to $\omega_k$ (the right hand side is also a polynomial in $\omega_k$), so we can set $\omega_k = \epsilon$ and take $\epsilon\to0$. Taking the $X_*$ expectation, we get
\beqs\label{final}\E_{X_*}\bigg[c_{k,k}\omega_k +& \sum_{\substack{\lambda\in S_{2k}\\\ell_{\max}=k}}c_{\lambda}\prod_{\ell\in\lambda\setminus p}\omega_\ell  \bigg] \\&= -2c_{k,k}\E_{X_*}\bigg[\frac{d^k}{dt^k}(1+q_\omega(t))^{-1}\big\vert_{t=0}\bigg]\eeqs
\item We show that 
$$\E_{X_*}\bigg[\frac{d^k}{dt^k}(1+q_\omega(t))^{-1}\big\vert_{t=0}\bigg] = \frac{d^k}{dt^k}\frac{1 + q^*(t)}{1 + q(t)}\bigg\vert_{t=0}.$$
\item We show that
$$\frac{d^k}{dt^k}\frac{1 + q^*(t)}{1 + q(t)}\bigg\vert_{t=0} = \mu_k^* - \mu_k + \pi(\mu_{1:k-1}, \mu_{1:k-1}^*),$$ where $\pi$ has the desired properties.  
\end{enumerate}

We prove {\bf A.} here and delegate the rest to the appendix; see Lemma~\ref{app:end-of-proof}. Consider the moment-cumulant relation $\kappa_{2k}(\omega_1, \dots, \omega_{2k}) = \sum_{\lambda\in S_{2k}}c_\lambda\prod_{\ell\in\lambda}\omega_\ell$. We write it as follows:
\beqs\label{kappatilde}
\kappa_{2k}(\omega_1, \dots, \omega_{2k})= \sum_{\substack{\lambda\in S_{2k}\\\lambda_{\mathrm{max}}> k}}c_{\lambda}\prod_{\ell\in\lambda}\omega_{\ell} +\omega_k\sum_{\substack{\lambda\in S_{2k}\\\lambda_{\mathrm{max}}=k}}c_{\lambda}\prod_{\ell\in\lambda\setminus k}\omega_{\ell} +\sum_{\substack{\lambda\in S_{2k}\\\lambda_{\mathrm{max}}<k}}c_{\lambda}\prod_{\ell\in\lambda}\omega_{\ell}
\eeqs
The sum in the middle contains the expression we are interested in. We will express it in terms of the function $\kappa_{2k}$ by strategically replacing some of the arguments with zero. Indeed, if we replace $\omega_\ell$ with zero for all $\ell > k$, then the first sum will vanish. If we also replace $\omega_k$ with zero, we get the third sum. The second sum is then the difference of these two. Summarizing, we have
\beqs \label{kappa-eqn}\omega_k\sum_{\substack{\lambda\in S_{2k}\\\lambda_{\mathrm{max}}=k}}c_{\lambda}\prod_{\ell\in\lambda\setminus k}\omega_{\ell} = \kappa_{2k}(\omega_1, \dots, \omega_k, 0, \dots, 0) - \kappa_{2k}(\omega_1, \dots, \omega_{k-1},0, 0, \dots, 0).\eeqs
Now, recall that $q_\omega(t) = \sum_{j=1}^k\omega_jt^j/j!$. Then by definition, 
\beqs\label{kappa-eqn-II}\kappa_{2k}(\omega_1, \dots, \omega_k) &= (2k)!\frac{d^{2k}}{dt^{2k}}\log\,\lp 1 +  q_\omega(t)\rp\big\vert_{t=0}\\
\kappa_{2k}(\omega_1, \dots, \omega_{k-1}) &= (2k)!\frac{d^{2k}}{dt^{2k}}\log\lp 1 + q_\omega(t)-\frac{\omega_kt^k}{k!}\rp\big\vert_{t=0},
\eeqs where we have omitted the zero arguments for brevity. Let $M(t) = 1-\frac{\omega_kt^k}{k!}(1+q_\omega(t))^{-1}$. Substituting \eqref{kappa-eqn-II} into \eqref{kappa-eqn}, we have
\beqs \omega_k\sum_{\substack{\lambda\in S_{2k}\\\lambda_{\mathrm{max}}=k}}c_{\lambda}\prod_{\ell\in\lambda\setminus k}\omega_{\ell} &= (2k)!\frac{d^{2k}}{dt^{2k}}\lb\log\,\lp 1+ q_\omega(t)\rp - \log\lp 1+q_\omega(t)-\frac{\omega_kt^k}{k!}\rp\rb\big\vert_{t=0}\\
&=-(2k)!\frac{d^{2k}}{dt^{2k}}\log\, M(t)\big\vert_{t=0}.
\eeqs By Proposition~\ref{cum-mom-prop} and since $M(0)=1$, this is given by  $\kappa_{2k}(M^{(1)}(0), \dots, M^{(2k)}(0))$.

\end{proof}


\bibliographystyle{alpha}
\bibliography{biblio}


\appendix

\section{Generalized Moment-Cumulant Relationship}\label{app:cum-mom}
Recall the definition of the polynomials $\kappa_m(\mu_1,\dots, \mu_m)$:
\beq\label{app:kappa-mu}
\kappa_m(\mu_1,\dots, \mu_m) = \sum_{\lambda\in S_m}c_{\lambda}\prod_{k\in\lambda}\mu_k,
\eeq
where $S_m$ is the set of all finite lists $\lambda$ of positive integers whose sum is $m$, and $c_{\lambda}\in\R$ are univeral constants.
\begin{prop}\label{app:prop:logM}
Let $M(t)$ be a real analytic function in the neighborhood $|t-t_0| < R$, for which $M(t_0)\neq 0$. If $$\sup_{|t-t_0|<R}\left|\frac{M(t)}{M(0)} - 1\right| < 1,$$ then 
$$\frac{d^m}{dt^m}\log M(t)\big\vert_{t=t_0} = \kappa_m\lp \frac{M^{(1)}(t_0)}{M(t_0)},\frac{M^{(2)}(t_0)}{M(t_0)}\dots, \frac{M^{(m)}(t_0)}{M(t_0)}\rp,$$ where the function $\kappa_m$ is defined by \eqref{app:kappa-mu}.

In particular, if $u(t)$ is given by the convergent series expansion 
$$u(t) = \sum_{k=1}^{\infty}\frac{\mu_k}{k!}(t-t_0)^k,\quad |t-t_0| < R,$$ and $$\sup_{|t-t_0|<R}|u(t)|< 1,$$ then
$$\log\lp 1+u(t) \rp = \sum_{m=1}^{\infty}\frac{\kappa_m}{m!}(t-t_0)^m, \quad |t-t_0| < R,$$ where the $\kappa_m$ are given by \eqref{app:kappa-mu}.
\end{prop}

\begin{proof}
We have
$$\frac{M(t)}{M(t_0)} = 1+\sum_{k=1}^{\infty}\frac{M^{(k)}(t_0)}{M(t_0)}\frac{(t-t_0)^k}{k!},\quad |t-t_0| < R.$$
Define $\mu_k = M^{(k)}(t_0)/M(t_0)$, and $u(t) = M(t)/M(t_0) - 1$, so that
$$u(t) = \sum_{k=1}^{\infty}\frac{\mu_k}{k!}(t-t_0)^k,\quad |t-t_0| < R.$$
If $\sup_{|t-t_0| < R}|u(t)| < 1$, then $f(t) = \log(M(t)/M(t_0)) = \log(1+u(t))$ is also real analytic in this neighborhood, and therefore has a convergent Taylor series expansion 
$$f(t) = \sum_{m=1}^{\infty}f^{(m)}(t_0)\frac{(t-t_0)^m}{m!},\quad |t-t_0| < R.$$ Note that $$f^{(m)}(t_0) = \frac{d^m}{dt^m}\log M(t)\big\vert_{t=t_0}.$$
On the other hand, since the series defining $u(t)$ is absolutely convergent and since $|u(t)|$ stays below $1$, we also have
\beqs
f(t) = \log(1+u(t)) &= \sum_{j=1}^{\infty}\frac{(-1)^{j-1}}{j}u(t)^j\\
&=  \sum_{j=1}^{\infty}\frac{(-1)^{j-1}}{j}\left(\sum_{k=1}^{\infty}\frac{\mu_k}{k!}(t-t_0)^k\right)^j\\
&= \sum_{m=1}^{\infty}\frac{\kappa_m}{m!}(t-t_0)^m,
\eeqs
where in the last line, the $\kappa_m$ are obtained by expanding the powers of the series and rearranging terms to combine like powers of $t-t_0$. Note that the lowest power of $t-t_0$ in $u(t)^k$ is $k$; therefore, the term $(t-t_0)^m$ arises only in the series expansions of $u(t)^1, \dots, u(t)^m$. Moreover, the coefficient of $(t-t_0)^m$ in the expansions of $u(t)^k,\,k=1,\dots,m$ will depend only on $\mu_1, \dots, \mu_m$. Therefore, $\kappa_m$ is a polynomial of $\mu_1,\dots,\mu_m$, and it is clear that this polynomial should be of the form\eqref{app:kappa-mu}. The coefficients $c_{\lambda}$ clearly do not depend on the particular values of the $\mu_k$, so they are universal constants. 

Summarizing, we have shown that
$$\frac{d^m}{dt^m}\log M(t)\big\vert_{t=t_0} = \kappa_m(\mu_1,\dots, \mu_m) = \kappa_m\lp \frac{M^{(1)}(t_0)}{M(t_0)},\frac{M^{(2)}(t_0)}{M(t_0)}\dots, \frac{M^{(m)}(t_0)}{M(t_0)}\rp.$$
The second assertion about the Taylor expansion of $\log(1+u(t))$ clearly follows. 
\end{proof}

\begin{lemma}\label{cMM}
The coefficient $c_{2p}$ in front of the term $\mu_{2p}$ in the polynomial defining $\kappa_{2p}$ (see \eqref{app:kappa-mu}) is $c_{2p}=1$. The coefficient $c_{p,p}$ in front of the term $\mu_p^2$ is given by $c_{p,p} = -\frac12{2p\choose p}.$
\end{lemma}
\begin{proof}
Let $\epsilon = \sum_{k=1}^{\infty}\mu_kt^k/k!$, so that $$\log(1+\epsilon) = \sum_{j=1}^{\infty}\frac{(-1)^{j+1}}{j!}\epsilon^j.$$ Note that the moments appear individually in $\epsilon$. Thus, the moment $\mu_{2p}$ by itself can only appear in $\epsilon^1$, where it has coefficient $1/(2p)!$. Since $\kappa_{2p}$ is $(2p)!$ times the $t^{2p}$ coefficient, we must have $c_{2p}=1$. 

Now, similarly, $\mu_p^2$ is a product of two moments and can therefore only appear in $\epsilon^2$. In the expansion of $\epsilon^2$, $\mu_p^2t^{2p}$ will appear with coefficient $1/(p!)^2.$ Multiplying by $-1/2$, we have that the coefficient of $\mu_p^2t^{2p}$ in the expansion of  $\log(1+\epsilon)$ is $-1/2(p!)^2$. Multiplying by $(2p)!$, we arrive at 
$$c_{p,p} = -\frac12\frac{(2p)!}{p!p!} = -\frac12{2p\choose p}.$$
\end{proof}

\begin{lemma}\label{app:end-of-proof}
Let $X, X_*\in\R$ be bounded random variables, and let $W=X-X_*$. We denote the moments of $X,X_*,W$ as
$$\mu_j = \E\lb X^j\rb, \quad \mu_j^* = \E\lb X_*^j\rb, \quad \omega_j = \E_X\lb W^j\rb.$$ Note that $\omega_j$ is random with respect to $X_*$. Let $c_\lambda$ be the universal coefficients of the moment-cumulant relationships ~\eqref{app:kappa-mu}. Then
\beq\label{app:XXstarW} \E_{X_*}\bigg[\bigg(\sum_{\substack{\lambda\in S_{2k}\\\ell_{\max}=k}}c_{\lambda}\prod_{\ell\in\lambda\setminus k}\omega_\ell\bigg) + c_{k,k}\omega_k\bigg] = 2c_{k,k}(\mu_k - \mu_k^*) + \pi(\mu_{1:k-1}, \mu_{1:k-1}^*),\eeq
where $\pi$ is a polynomial in $\mu_{1:k-1}, \mu_{1:k-1}^*$ with universal coefficients (independent of the values of the moments) such that each monomial has total order $k$ and which satisfies $\pi(\mu_{1:k-1}^*, \mu_{1:k-1}^*)=0$. 
\end{lemma}
Define $$q(t) =  \sum_{j=1}^k\mu_jt^j/j!,\quad q_*(t) =  \sum_{j=1}^k\mu^*_jt^j/j!, \quad q_\omega(t) =  \sum_{j=1}^k\omega_jt^j/j!.$$ The proof consists of the following steps (we only summarize them, for more detail see the main text).
\begin{enumerate}[wide=0pt,label=\textbf{\Alph*.}]
\item We showed in the main text that
$$\omega_k\bigg( c_{k,k}\omega_k + \sum_{\substack{\lambda\in S_{2k}\\\ell_{\max}=k}}c_{\lambda}\prod_{\ell\in\lambda\setminus p}\omega_\ell\bigg) =c_{k,k}\omega_k^2  -\kappa_{2k}(M^{(1)}(0), \dots, M^{(2k)}(0)),$$ where
$$M(t) = 1-\frac{\omega_kt^k}{k!}(1+q_\omega(t))^{-1}.$$
\item We show that
$$c_{k,k}\omega_k^2 - \kappa_{2k}(M^{(1)}(0), \dots, M^{(2k)}(0))= -2c_{k,k}\omega_k\frac{d^k}{dt^k}(1+q_\omega(t))^{-1}\big\vert_{t=0},$$
\item We show that 
$$\E_{X_*}\bigg[\frac{d^k}{dt^k}(1+q_\omega(t))^{-1}\big\vert_{t=0}\bigg] = \frac{d^k}{dt^k}\frac{1 + q^*(t)}{1 + q(t)}\bigg\vert_{t=0}.$$
\item We show that
$$\frac{d^k}{dt^k}\frac{1 + q^*(t)}{1 + q(t)}\bigg\vert_{t=0} = \mu_k^* - \mu_k + \pi(\mu_{1:k-1}, \mu_{1:k-1}^*),$$ where $\pi$ has the desired properties.  
\end{enumerate}
In Lemma~\ref{lma:toFinishForReal}, we proved {\bf A.}, so it remains to prove {\bf B., C., D.}
\begin{proof}
For {\bf B.}, note that $M^{(p)}(0) = 0,p=1,\dots, k-1$, so all products $\prod_{\ell\in\lambda}M^{(\ell)}(0)$ involving $\ell < k$ are zero. But the only lists $\lambda\in S_{2k}$ involving only moments of order $k$ and higher are $\lambda = \{k,k\}$ and $\lambda = \{2k\}$. We must therefore compute the $k$th and $2k$th derivatives of $M$ at zero. By the product rule, and using that only the $k$th derivative of $t^k$ is nonzero at $t=0$, we have $$M^{(k)}(0) = -\omega_kq^{-1}(0) =  -\omega_k,\quad M^{(2k)}(0) = -\omega_k{2k\choose k}\frac{d^k}{dt^k}(1+q_\omega(t))^{-1}\big\vert_{t=0}.$$ Therefore, 
\beqs \kappa_{2k}(M^{(1)}(0), \dots, M^{(2k)}(0))  &= c_{2k}M^{(2k)}(0) + c_{k,k}\lp M^{(k)}(0)\rp^2 \\ &=  -c_{2k}\omega_k{2k\choose k}\frac{d^k}{dt^k}(1+q_\omega(t))^{-1}\big\vert_{t=0} + c_{k,k}\omega_k^2 \\&= 2c_{k,k}\omega_k\frac{d^k}{dt^k}(1+q_\omega(t))^{-1}\big\vert_{t=0} +  c_{k,k}\omega_k^2\eeqs (We used that $c_{2k}=1$ and ${2k\choose k} = -2c_{k,k}$, proved in Lemma~\ref{cMM}). Subtracting $c_{k,k}\omega_k^2$ and multiplying by $-1$ gives the desired result.

This concludes {\bf B.} For {\bf C.}, define 
$$G(t) = \E\lb e^{Xt}\rb,\quad G_*(t) = \E\lb e^{X_*t}\rb,\quad G_\omega(t) = \E_X\lb e^{Wt}\rb.$$
Note that $G_{\omega}(t) = 1 + q_\omega(t) + t^{k+1}r_\omega(t)$ for some smooth function $r_\omega(t)$. By Taylor expanding $(1+q_\omega)^{-1}$ and $(1 + q_\omega + t^{k+1}r_\omega)^{-1}$ in a neighborhood of $0$, we get
$$(1+q_\omega)^{-1} = 1 + \sum_{p=1}^\infty(-1)^pq_\omega^p,\quad G_{\omega}^{-1} = 1 + \sum_{p=1}^{\infty}(-1)^p(q_\omega + t^{k+1}r_\omega)^p.$$ Combining like powers of $t$ (justified by the absolute convergence of both series in a neighborhood of zero), we see that the order $t^k$ term in both series are the same, and hence $\frac{d^k}{dt^k}(1+q_\omega(t))^{-1}\big\vert_{t=0} = \frac{d^k}{dt^k}G_\omega(t)^{-1}\big\vert_{t=0}.$ We now take the $X_*$ expectation and bring it inside the derivative. This is justified by the absolute convergence of the series and its derivatives, and the fact that $X_*$ is bounded. Now, note that since $W=X-X_*$, we can write $G_\omega(t)  =  e^{-tX_*}G(t)$, so that $$\E_{X_*}\lb G_\omega(t)^{-1}\rb = G^*(t)/G(t).$$ Finally, 
$G^*/G = (1+q^* + t^{k+1}r^*)(1+q + t^{k+1}r)^{-1}$ for some smooth function $r^*$ and $r$. By a similar Taylor expansion argument as before, the $t^k$ coefficient of the expansion of $G^*/G$ around $t=0$ is the same as that of $(1+q^*)/(1+q)$.\\

This concludes {\bf C.}, and we turn to {\bf D.} We have
\beqs
\frac{1+q^*}{1+q} &= 1 + (q^* - q) + (q^* -q)\lp\sum_{p=1}^k(-q)^p + t^{k+1}c(t)\rp\\
&= \mathrm{l.o.t.} + \frac{\mu_k^* - \mu_k}{k!}t^k + \sum_{\ell=1}^k\frac{\mu_\ell^* - \mu_\ell}{\ell!}t^\ell\sum_{p=1}^k(-q)^p  +\mathrm{h.o.t.}
\eeqs
where $t^{k+1}c(t)$ denotes the higher order terms in the expansion of $(1+q)^{-1}$, l.o.t. denotes terms $t^j$ for $j<k$ and h.o.t. denotes terms $t^j$ for $j>k$. Now, note that the lowest order power of $t$ appearing in $\sum_{p=1}^k(-q)^p$ is $t^1$. Hence, the product $(\mu_k^*-\mu_k)t^k\sum_{p=1}^k(-q)^p$ only involves terms $t^j$, $j>k$, so we can discard $\ell = k$ from the sum. For each $\ell = 1, \dots, k-1$, we collect the terms involving $t^{k-\ell}$ in $\sum_{p=1}^k(-q)^p$. Since the coefficient of $t^j$ in $q$ is $\mu_j/j!$, we collect products of $\mu_j$'s with total moment order $k-\ell$. Hence, the sum of all the $t^{k-\ell}$ coefficients appearing in $\sum_{p=1}^k(-q)^p$ can be written in the form
$\sum_{\lambda\in S_{k-\ell}}d_{\lambda}\prod_{\ell'\in\lambda}\mu_{\ell'}.$ Combining these observations, we have
\beqs
\frac{1+q^*}{1+q} = \mathrm{l.o.t.} + \bigg(\mu_k^* - \mu_k + \sum_{\ell = 1}^{k-1}\lp\mu_{\ell}^* - \mu_{\ell}\rp\sum_{\lambda\in S_{k-\ell}}d_{\lambda}\prod_{\ell'\in\lambda}\mu_{\ell'} \bigg)\frac{ t^k}{k!} + \mathrm{h.o.t.},
\eeqs
where the expression in parenthesis is the $k$th derivative of $(1+q^*)/(1+q)$ at zero, and
$$\pi = \sum_{\ell = 1}^{k-1}\lp\mu_{\ell}^* - \mu_{\ell}\rp\sum_{\lambda\in S_{k-\ell}}d_{\lambda}\prod_{\ell'\in\lambda}\mu_{\ell'}.$$
It is clear to see that $\pi$ depends only on $\mu_{1:k-1}, \mu_{1:k-1}^*$ and that $\pi(\mu_{1:k-1}^*, \mu_{1:k-1}^*) = 0$, and that the total moment order of each term is $k$. 
\end{proof}
Recall that $w = \ta-\ta_*$, where $\ta\sim\rho$ and $\ta_*\sim\rho_*$ have compact support. Recall also the moment tensors
\beq
T_k = \E_{\ta\sim\rho}\lb\ta^{\otimes k}\rb,\quad T_k^* = \E_{\ta_*\sim\rho_*}\lb\ta_*^{\otimes k}\rb.
\eeq
and the space $V_k\lb T_{1:k-1}, T_{1:k-1}^*\rb$ of tensors given by sums of tensor products of $T_i, T_j^*, i,j<k$ where each product has total moment order $k$.  Finally, recall the tensor
$$J_k = \sum_{\substack{\lambda\in S_{2k}\\\ell_{\max}=k}}c_{\lambda}\bigotimes_{\ell\in\lambda\setminus k}\E_{\ta}\lb w^{\otimes \ell}\rb.$$
\begin{lemma}\label{app:equiv}
Lemma~\ref{app:end-of-proof} implies that
$$\E_{\ta_*}\bigg\la x^{\otimes k}, \; J_k + c_{k,k}\E_{\ta}\lb w^{\otimes k}-\ta^{\otimes k}\rb\bigg\ra = \bigg\la x^{\otimes k},\; -2c_{k,k}T_k^* + c_{k,k}T_k + Q_k\bigg\ra,$$ where $Q_k\in V_k\lb T_{1:k-1}, T_{1:k-1}^*\rb$ is $x$-independent.
\end{lemma}
\begin{proof}
Take $X=x^T\ta$, $X_* = x^T\ta_*$, $W = x^T(\ta-\ta_*) = x^Tw$. Then \beq\label{mumustaromega}\mu_j = \lla x^{\otimes j},\;T_j\rra ,\quad \mu_j^*= \lla x^{\otimes j},\; T_j^*\rra ,\quad \omega_j = \lla x^{\otimes j},\; \E_\ta\lb w^{\otimes j}\rb\rra.\eeq  Using \eqref{app:XXstarW}, we then have
\beqs\label{xJc}
\E_{\ta_*}\bigg\la x^{\otimes k}, \; J_k &+ c_{k,k}\E_{\ta}\lb w^{\otimes k}-\ta^{\otimes k}\rb\bigg\ra \\&=  \E_{X_*}\bigg[\bigg(\sum_{\substack{\lambda\in S_{2k}\\\ell_{\max}=k}}c_{\lambda}\prod_{\ell\in\lambda\setminus k}\omega_\ell\bigg) + c_{k,k}(\omega_k - \mu_k)\bigg]\\
&=\E_{X_*}\bigg[\bigg(\sum_{\substack{\lambda\in S_{2k}\\\ell_{\max}=k}}c_{\lambda}\prod_{\ell\in\lambda\setminus k}\omega_\ell\bigg) + c_{k,k}\omega_k\bigg] - c_{k,k}\mu_k\\
&= c_{k,k}\mu_k - 2c_{k,k}\mu_k^* + \pi(\mu_{1:k-1}, \mu_{1:k-1}^*)\\
&= \bigg\la x^{\otimes k},\; -2c_{k,k}T_k^* + c_{k,k}T_k\bigg\ra + \pi(\mu_{1:k-1}, \mu_{1:k-1}^*).
\eeqs
Now, we show in Lemma~\ref{app:end-of-proof} that $\pi$ has the form 
$$\pi(\mu_{1:k-1}, \mu_{1:k-1}^*) = \sum_{\ell = 1}^{k-1}\lp\mu_{\ell}^* - \mu_{\ell}\rp\sum_{\lambda\in S_{k-\ell}}d_{\lambda}\prod_{\ell'\in\lambda}\mu_{\ell'}.$$
Using the representation ~\eqref{mumustaromega}, we can write $\pi$ as 
\beqs
\pi(\mu_{1:k-1}, \mu_{1:k-1}^*) &= \sum_{\ell = 1}^{k-1}\lla x^{\otimes \ell},\;T_\ell^*-T_\ell\rra\sum_{\lambda\in S_{k-\ell}}d_{\lambda}\prod_{\ell'\in\lambda}\lla x^{\otimes \ell'},\;T_{\ell'}\rra\\
&=  \lla x^{\otimes k},\; \sum_{\ell = 1}^{k-1}(T_\ell^*-T_\ell)\otimes\sum_{\lambda\in S_{k-\ell}}d_{\lambda}\bigotimes_{\ell'\in\lambda}T_{\ell'}\rra
\eeqs
The tensor on the right hand side is the desired $Q_k$. We clearly have $Q(T_{1:k-1}=T_{1:k-1}^*, T_{1:k-1}^*)=0$.
\end{proof}
\section{Moment Tensor Computations}
Let $\ta, \ta_*\in\R^d$ be random vectors distributed according to distributions $\rho$ and $\rho_*$ respectively, each of which is a compactly supported probability measure. Recall the moment tensors
\beqs
T_k &= T_k(\rho) = \E_{\ta\sim\rho}\lb\ta^{\otimes k}\rb,\\
T_k^* &= T_k(\rho_*) = \E_{\ta_*\sim\rho_*}\lb\ta_*^{\otimes k}\rb.
\eeqs
For a multi-index $\ell = (\ell_1, \dots, \ell_{k})$, we write $T_k^\ell$ to denote the tensor entry $T_k^{\ell_1,\dots, \ell_k}$. 

Recall the spaces $V_k, R_k$ from Definition~\ref{VkRk}. We define these spaces more formally here. 
\begin{defn}
Consider finite lists $I, J$ of positive integers, which may include several of the same number, satisfying
\beq\label{app:IJ}\max_{i\in I}i \leq m,\,\;\max_{j\in J}j \leq n,\;\sum_{i\in I}i + \sum_{j\in J}i = k.\eeq
Consider the set of tensors $S_{k,m,n} = \left\{\prod_{i\in I}T_i\otimes \prod_{j\in J}T_j^*\,\bigg\vert\,I,J\,\mathrm{satisfy}\,\eqref{app:IJ}\right\}$. We define the following spaces
\beqs
V_k\lb T_{1:m}, T_{1:n}^*\rb &=\mathrm{span}\; S_{k,m,n},\\
R_k\lb T_{1:m}, T_{1:n}^*\rb &= \mathrm{span}\left\{ \la A, T\ra\mid A\in{\R^d}^{\otimes k}, T\in S_{k,m,n}\right\}
\eeqs

\end{defn}
\begin{example}
The set $S_{2,1,1}$ is given by $S_{2,1,1} = \{T_1\otimes T_1, T_1\otimes T_1^*, T_1^*\otimes T_1\}.$ Hence, tensors in the space $V_2\lb T_1, T_1^*\rb$ are of the form
$$aT_1\otimes T_1 + bT_1\otimes T_1^* + cT_1^*\otimes T_1*$$ while polynomials in the space $R_2\lb T_1, T_1^*\rb$ are of the form
$$\lla A, T_1\otimes T_1\rra +\lla B, T_1\otimes T_1^*\rra + \lla C, T_1^*{\otimes}T_1^*\rra.$$
\end{example}
We begin with a result about moment tensors for the orbit recovery model. 


\begin{prop}\label{app:orbit-recovery}
Let $G\subset O(d)\subset R^{d\times d}$ be a subgroup of the orthogonal group in dimension $d$, and $\ta,\ta_*\in\R^d$ be deterministic vectors. Let $\gamma$ be a Haar measure on $G$, and let $T_k, T_k^*$ be the moment tensors of the distributions ${\bf g}\ta,\;{\bf g}\ta_*,\;{\bf g}\sim\gamma$, i.e.
\beq\label{orbit-dist}T_k = \E_{{\bf g}\sim\gamma}\lb\lp{\bf g}\ta\rp^{\otimes k}\rb,\quad T_k^* = \E_{{\bf g}\sim\gamma}\lb\lp{\bf g}\ta_*\rp^{\otimes k}\rb.\eeq Then for tensors $Q \in V_k\lb T_{1:m}, T_{1:n}^*\rb$, we have
\beq\label{orbit-eqn}\lla T_k, Q\rra\in R_{2k}\lb T_{1:m}, T_{1:n}^*\rb.\eeq 
\end{prop}
The proposition will follow from the following lemma. 
\begin{lemma}\label{app:orbit-recovery-lma}
Let $I, J$ satisfy \eqref{app:IJ}. Then for moment tensors \eqref{orbit-dist}, we have
\beq\label{app:orbit-lemma}
\lla T_{k},\; \prod_{i\in I}T_i\otimes\prod_{j\in J}T_j^*\rra = \prod_{i\in I}\lla T_i,\; T_i\rra\prod_{j\in J}\lla T_j,\;T_j^*\rra.
\eeq
\end{lemma}
\begin{proof}[Proof of Proposition~\ref{app:orbit-recovery}]
To prove the proposition, it suffices to show $\la T_k, Q\ra\in R_{2k}\lb T_{1:m}, T_{1:n}^*\rb$ for $Q\in S_{k,m,n}$. Let $Q = \prod_{i\in I}T_i\otimes\prod_{j\in J}T_j^*$. By the lemma, we have 
$$\lla T_k, Q\rra = \prod_{i\in I}\lla T_i,\; T_i\rra\prod_{j\in J}\lla T_j,\;T_j^*\rra.$$

Now, for an order $2p$ multi-index $\ell$, write $\ell = (\ell_1, \ell_2)$, where $\ell_1$ and $\ell_2$ are order $p$ multi-indices. Define the order $2p$ tensor $M_{2p}$ by $M_{2p}^{\ell_1, \ell_2} = 1$ if $\ell_1 = \ell_2$ and $M_{2p}^{\ell_1, \ell_2}=0$ otherwise. We can then write
$\la T_i, T_i\ra = \la M_{i}, T_i\otimes T_i\ra$ and $\la T_j, T_j^*\ra = \la M_{j}, T_j\otimes T_j^*\ra.$
Thus, 
\beqs
\prod_{i\in I}\lla T_i,\; T_i\rra\prod_{j\in J}\lla T_j,\;T_j^*\rra &=\prod_{i\in I}\lla M_{2i}, \; T_i\otimes T_i\rra\prod_{j\in J}\lla M_{2j},\;T_j\otimes T_j^*\rra \\
&= \lla\prod_{i\in I}M_{2i}\otimes\prod_{j\in J}M_{2j},\;\prod_{i\in I}T_i\otimes T_i\otimes\prod_{j\in J}T_j\otimes T_j^*.\rra 
\eeqs
The tensor product on the right has total order $2k$ and involves only the tensors $T_{1:m}$, $T_{1:n}^*$. It therefore lies in $S_{2k,m,n}$, so its inner product with a constant tensor belongs to $R_{2k}\lb T_{1:m}, T_{1:n}^*\rb$.
\end{proof}
\begin{proof}[Proof of Lemma~\ref{app:orbit-recovery-lma}]
Since the elements $g$ of the group $G$ act on vectors by orthogonal transformation, we have
$$\lla g\ta,\; h\ta\rra =\lla \ta,\; g^{-1}h\ta\rra =  \lla g'\ta,\; g'g^{-1}h\ta\rra\;\forall g, g',h\in G.$$ This implies 
\beqs
\lla \lp g\ta\rp^{\otimes \ell},\; T_\ell\rra &= \E_{h\sim\rho}\lla g\ta, h\ta\rra^\ell=\E_{h\sim\rho}\lla g'\ta, g'g^{-1}h\ta\rra^\ell \\&=  \lla (g'\ta)^{\otimes \ell}, \E_{h\sim\rho}\lb \lp g'g^{-1}h\ta\rp^{\otimes\ell}\rb\rra = \lla (g'\ta)^{\otimes \ell},\; T_\ell\rra\;\forall g,g'\in G,
\eeqs
since the measure $\rho$ on $G$ is Haar and therefore invariant under multiplication by group elements. 
Similarly, $\la \lp g\ta\rp^{\otimes \ell},\; T_\ell^*\ra = \la \lp g'\ta\rp^{\otimes \ell},\; T_\ell^*\ra\;\forall g,g'\in G$. 
Averaging over $g'\in G$ then gives
$$\lla \lp g\ta\rp^{\otimes \ell},\; T_\ell\rra = \lla T_\ell,\; T_\ell\rra,\quad \lla \lp g\ta\rp^{\otimes \ell},\; T_\ell^*\rra = \lla T_\ell,\; T_\ell^*\rra\;\forall g\in G.$$ We therefore have
\beqs
\lla T_k,\; \prod_{i\in I}T_{i}\otimes\prod_{j\in J}T_j^*\rra &=  \E_{g\sim\rho}\lla (g\ta)^{\otimes k},\; \prod_{i\in I}T_{i}\otimes\prod_{j\in J}T_j^*\rra\\
&=\E_{g\sim\rho}\prod_{i\in I}\lla (g\ta)^{\otimes i},\; T_{i}\rra\prod_{j\in J}\lla (g\ta)^{\otimes j},\;T_j^*\rra\\
&=\prod_{i\in I}\lla T_{i},\; T_{i}\rra\prod_{j\in J}\lla T_j,\;T_j^*\rra.
\eeqs
\end{proof}
We return now to the general setting in which $\ta\sim\rho, \ta_*\sim\rho_*$ are random. Recall that the entries of $\ta$ are denoted with superscripts, i.e.  $\ta = (\ta^1, \dots, \ta^d)$ and similarly for $\ta_*$. 
\begin{lemma}\label{lemma:thstar-expect}
Let $v_1,\dots,v_n$ be multi-indices, where the indices in each multi-index are between $1$ and $d$. Let $|v_j|$ denote the number of indices in $v_j$. 
Define $L = \max\{|v_1|,|v_2|,\dots, |v_n|\}$ and $M = |v_1| + |v_2| + \dots + |v_n|$, and let
\beq\label{lem:prod}
a(\ta, \ta_*) = \E_{\ta_*}\prod_{j=1}^n\E_\ta\lb\prod_{k\in v_j}(\ta^k-\ta_*^k)\rb
\eeq
Then $a(\ta, \ta_*)\in R_M\lb T_{1:L}, T_{1:M}^*\rb.$
\end{lemma}
\begin{proof}
For an arbitrary multi-index $v$, we have
\beq\label{T_nI}\begin{split}
\E_{\ta}\left[\prod_{k\in v}(\ta^{k}-\ta_*^k)\right]&=\E_{\ta}\left[\sum_{I\subset v}\prod_{i\in I}\ta^{i}(-1)^{|v|-|I|}\prod_{j\in v\setminus I}\ta_*^j\right]\\
&= \sum_{I\subset v}\E_{\ta}\left[\prod_{i\in I}\ta^{i}\right](-1)^{|v|-|I|}\prod_{j\in v\setminus I}\ta_*^j\\
&= \sum_{I\subset v}T_{|I|}^{I}(-1)^{|v|-|I|}\left(\prod_{j\in v\setminus I}\ta_*^j\right).
\end{split}\eeq
Note that $|I| \leq |v|$. Now, applying \eqref{T_nI} for $v = v_1,\dots, v_n$ and multiplying the results together yields entries of $\ta$-moment tensors of order no higher than $L=\max_{1\leq i\leq n}|v_i|$ as well as products of at most $M = |v_1| + \dots + |v_n|$ entries of $\ta_*$. Upon taking the $\ta_*$-expectation, these products will become entries of $\ta_*$-moment tensors of order at most $M$. Noting that the sum of the moment orders in the product \eqref{lem:prod} is $M$, we see that $a(\ta, \ta_*) \in R_M\lb T_{1:L}, T_{1:M}^*\rb$, as desired. 
\end{proof}

Recall the definition
$$J_p = \sum_{\substack{\lambda\in S_{2p}\\\ell_{\max}=p}}c_{\lambda}\bigotimes_{\ell\in\lambda\setminus p}\E_{\ta}\lb w^{\otimes \ell}\rb,$$ where $w = \ta - \ta_*$.  
\begin{lemma}\label{r-nu}
Define
$$r(\ta, \ta_*) = \E_{\ta_*}\bigg\la \E_\ta\lb w^{\otimes p}\rb - T_p,\; J_p - c_{p,p}T_p\bigg\ra.$$
Then $r(\ta, \ta_*) \in R_{2p}\lb T_{1: p-1}, T_{1:2p}^*\rb.$
\end{lemma}
\begin{proof}
Define $Q_p = \E_\ta\lb w^{\otimes p}\rb - T_p$ and $Q_p' = J_p - c_{p,p}\E_\ta\lb w^{\otimes p}\rb$, so that 
$$\bigg\la \E_\ta\lb w^{\otimes p}\rb - T_p,\; J_p - c_{p,p}T_p\bigg\ra = \lla Q_p, Q_p' + c_{p,p}Q_p\rra.$$ Let $v$ be a multi-index with $|v|=p$. By \eqref{T_nI} of Lemma~\ref{lemma:thstar-expect},we have
\beq
Q_p^v = \E_\ta\lb\prod_{k\in v}(\ta^k-\ta_*^k) - \prod_{k\in v}\ta^k\rb=\sum_{I\subsetneqq v}T_{|I|}^{I}(-1)^{|v|-|I|}\left(\prod_{i\in v\setminus I}\ta_{*i}\right).
\eeq
Thus, we have
\beqs
\E_{\ta_*}\lla Q_p, Q_p \rra &= \sum_{v}\sum_{I,J\subsetneqq v}c_{I,J}T_{|I|}^IT_{|J|}^J\E_{\ta_*}\prod_{i\in v\setminus I}\ta_{*i}\prod_{j\in v\setminus J}\ta_{*j}\\
&= \sum_{v}\sum_{I,J\subsetneqq v}c_{I,J}T_{|I|}^IT_{|J|}^JT_{2p-|I|-|J|}^{*v\setminus I, v\setminus J}\\
&= \sum_{v}\sum_{I,J\subsetneqq v}\lla \delta_{I,J,v},\;T_{|I|}\otimes T_{|J|}\otimes T_{2p-|I|-|J|}^*\rra,
\eeqs
where $\delta_{I,J,v}\in \lp\R^{d}\rp^{\otimes 2p}$ is such that $\delta_{I,J,v}^{\ell} = 1$ if $\ell = (I, J, v\setminus I, v\setminus J)$ and $\delta_{I,J,v}^{\ell} = 0$ otherwise. Thus, 
$$\E_{\ta_*}\lla Q_p, Q_p \rra \in R_{2p}\lb T_{1: p-1}, T_{1:2p}^*\rb.$$
Now, $Q_p'$ is given by
\beq
Q_p' = J_p - c_{p,p}\E_\ta\lb w^{\otimes p}\rb = \sum_{\substack{\lambda\in S_{2p}\\\lambda_{\max}=p,\\ \max(\lambda\setminus p) < p}}\bigotimes_{\ell\in\lambda\setminus p}\E_\ta\lb w^{\otimes \ell}\rb
\eeq
Similarly to $Q_p$, the entries of $Q_p'$ are given by sums of products of at most $p$ entries of $\ta_*$ with products of entries of $T_k$ for $k\leq p-1$. We can therefore similarly show that $\E_{\ta_*}\lla Q_p, Q_p' \rra \in R_{2p}\lb T_{1: p-1}, T_{1:2p}^*\rb.$ 
\end{proof}
Recall that to finish the proof of Proposition~\ref{EZW}, we need to compute
$$\sum_{\substack{\lambda\in S_{2k}\\\ell_{\max}=k}}c_{\lambda}\E_W\lb\E_{\ta}\lb(w^TW)^k\rb\prod_{\ell\in\lambda\setminus k}\E_{\ta}\lb(w^T\overline W)^{\ell}\rb\rb,$$ where $W = Z+iZ'$, and $Z,Z'\in\R^d$ are i.i.d. standard normal vectors. We compute the $W$-expectation of an arbitrary summand in the next proposition. Note that if $\lambda \in S_{2k}$ and $\ell_{\max}=k$, then $\lambda\setminus k\in S_k$. We therefore rename $\lambda\setminus k$ as $\lambda\in S_k$. 
\begin{prop}\label{app:ZexpectFinal}
Let $W = Z + iZ'$, where $Z,Z'\in\R^d$ are i.i.d. standard normal random vectors, and let $\lambda\in S_k$. 
We have
\beq\label{app:EZkappa}
	\E_W\left[\E_\ta\lb (w^TW)^k\rb\prod_{\ell\in\lambda}\E_\ta\lb (w^T\overline W)^\ell\rb\right]= 2^{k}k!\lla \E_\ta\lb w^{\otimes k}\rb, \; \bigotimes_{\ell\in\lambda}\E_\ta\lb w^{\otimes \ell}\rb\rra.
\eeq
\end{prop}
\begin{proof}
First, let $w_{(\ell)} = \ta_{(\ell)}-\ta_*,$ where $\ta_{(\ell)}\sim\rho,\;\ell\in\lambda\cup\{k\}$ are i.i.d. and $\ta_*$ is considered fixed throughout the proof. Define $V_{(\ell)} = w_{(\ell)}^TW, \;\ell\in\lambda\cup\{k\}$. Note that $(V_{(k)};\,(V_\ell)_{\ell\in\lambda})$ is jointly circularly symmetric for any fixed values of $w_{(k)},w_{(\ell)},\ell\in\lambda$ and that $\E_W[V_{(\ell)}\overline V_{(\ell')}] = 2w_{(\ell)}^Tw_{(\ell')},\;\ell,\ell'\in \lambda\cup\{k\}$. We write 
\beqs
\E_W\left[\E_\ta\lb (w^TW)^k\rb\prod_{\ell\in\lambda}\E_\ta\lb (w^T\overline W)^\ell\rb\right] &=\E_W\left[\E_{\ta_{(k)}}\lb (w_{(k)}^TW)^k\rb\prod_{\ell\in\lambda}\E_{\ta_{(\ell)}}\lb (w_{(\ell)}^T\overline W)^\ell\rb\right]\\
&= \E_{\ta_{(k)},\ta_{(\ell)},\ell\in\lambda}\,\E_W\bigg[V_{(k)}\prod_{\ell\in\lambda}\overline V_{(\ell)}\bigg].\eeqs
Using a result in \cite{moments}, we have that
\beqs\label{app:EZkappa-intermed}
\E_W\bigg[V_{(k)}\prod_{\ell\in\lambda}\overline V_{(\ell)}\bigg] &= k!\prod_{\ell\in\lambda}\E_W\left[V_{(k)}\overline V_{(\ell)}\right]^{\ell}= 2^kk!\prod_{\ell\in\lambda}(w_{(k)}^Tw_{(\ell)})^{\ell} \\
&= 2^kk!\lla w_{(k)}^{\otimes k}, \bigotimes_{\ell\in\lambda}w_{(\ell)}^{\otimes \ell}\rra.\eeqs
We take the expectation with respect to $\ta_{(k)},\ta_{(\ell)},\ell\in\lambda$, and bring the expectations inside the products to conclude. 
\end{proof}

\section{Estimates for EM}
Let $\bm\chi$ be a random variable supported in a set $X$. We denote samples from $\bm\chi$ by $\chi$. Let $\ta =(\ta_\chi)_{\chi\in X},$ where $\ta_\chi\in\R^d,\;\chi\in X$. For the purposes of this section $X$ need not be finite. We define
$\|\ta\|_\infty = \sup_{\chi\in X}\|\ta_\chi\|$, and $T_k(\ta) = \E_{\bm\chi\sim\rho}\lb\ta_{\bm\chi}^{\otimes k}\rb$. Recall from Section~\ref{sec:EM} that we may write the ground truth GMM as 
$$Y = \sg Z + \ta_{*\bm\chi}, \quad\bm\chi\sim\rho.$$ We will assume $\|\ta_*\|_\infty=1$ and that $\rho$ is known. Recall that the standard EM update is given by $\ta_\chi^{(t+1)} = G_\chi(\ta^{(t)})$, where
$$G_\chi(\ta) = \frac{\E_Y\lb w_\ta(Y,\chi)Y\rb}{\E_Y\lb w_\ta(Y,\chi)\rb},$$ and
$$
w_\ta(Y,\chi) = g\lp \frac{Y -\ta_{\chi}}{\sg}\rp\bigg/\E_{\bm\chi\sim\rho}\lb g\lp \frac{Y -\ta_{\bm\chi}}{\sg}\rp\rb
$$
Here, $g$ denotes the standard normal density in $\R^d$. We will write $G$ to denote $(G_\chi)_{\chi\in X}$.  In the following two lemmas and proposition, we let $R>0$ be constant, and assume $\|\ta_*\|_\infty=1$.
\begin{lemma}\label{lem:weights-expansion}
For all $\ta$ such that $(\|\ta\|_\infty\vee 1)/\sg\leq R$ and for all $\chi\in X$, we have
$$\E_Y\lb w_\ta(Y, \chi)\rb = 1 + \epsilon_1(\chi, \ta, \ta_*),$$ where
$$|\epsilon_1(\chi, \ta, \ta_*)|\leq C\lp\frac{\|\ta\|_\infty\vee 1}{\sg}\rp^2,$$
and $C$ depends on $R$ and $d$ only. 
\end{lemma}
\begin{proof} Let $\bm\chi'$ be an independent copy of $\bm\chi$. Analogously to the proof of the main theorem, we write $Y = \sg Z + \ta_{*\bm\chi'}$ . Since $\ta,\ta_*,\chi$ are fixed throughout the proof, and since $Y$ depends on $\sg$, $Z$, and $\bm\chi'$, we rename $w_\ta(Y, \chi)$ as $f(t, Z, \bm\chi')$, where $t=1/\sg$. Let $\delta = \sup_{\chi,\chi'\in X}\|\ta_\chi - \ta_{*\chi'}\|$, so that $\delta/\sg \leq 2R$. Also, define $v(\bm\chi,\bm\chi') = \ta_{\bm\chi}-\ta_{*\bm\chi'}$. (This corresponds to $w = \ta-\ta_*$ in the proof of the main theorem) and note that $\|v\|\leq\delta$. Now, let 
$$p(t,v, Z) = (Z^Tv)t - \frac{\|v\|^2}{2}t^2,$$ and note that $$g(Z - v/\sg) = \e\,\lp-\frac12\|Z\|^2\rp\e\, p(1/\sg,v, Z).$$ We then have
$$f(t, Z, \bm\chi') = \e\;p\bigg(t, v(\chi,\bm\chi'),Z\bigg)\bigg/\E_{\bm\chi\sim\rho}\bigg[\e\; p\bigg( t, v(\bm\chi,\bm\chi'),Z\bigg)\bigg].$$ 

 Let $M(t, Z, \bm\chi')$ denote the denominator of $f$,  so that $f = e^p/M$. We expand $f$ around $t=0$, and take its expectation with respect to $Z, \bm\chi'$. In the following, we let $f', f''$ denote the first and second partial derivative of $f$ with respect to $t$, respectively; $t$-derivatives of $p$ and $M$ are written analogously. We also suppress the arguments $Z, \bm\chi'$ on the right hand side. We have
\beq\label{EYft}\E_{Z,\bm\chi'}[f(t, Z, \bm\chi')] = 1+ t\E_{Z,\bm\chi'}[f'(0)] + \frac12t^2\E_{Z,\bm\chi'}[f''(\xi)],\quad |\xi|\leq t.\eeq Now, $f'= f(p' - M'/M)$. Note that $p' = Z^Tv -\|v\|^2t$ and $M' = \E_\chi\lb e^{p}p'\rb.$ Hence, $p'(0) = Z^Tv$, $M(0) = 1$, and $M'(0) = Z^T\E_\chi[v].$ Combining, we see that $f'(0) = Z^T(v- \E_\chi[v])$, which has zero $Z$-expectation. We now compute the second derivative of $f$. We have
$$
f'' = f\lp p'' - M''/M + (M'/M)^2\rp + f\lp p' - (M'/M)\rp^2, 
$$ so that
$$
|f''| \leq |p''| + |M''/M| + 3|M'/M|^2 + 2|p'|^2
$$
since $|f| < 1$. Recall from \eqref{M-bound} that
\beqs
\left|\frac{\partial_t^\ell M(\xi)}{M(\xi)}\right| \leq \delta^\ell\E_{Z'}(\|W\| + \delta|t|)^\ell,
\eeqs where $W = Z + iZ'$ and $Z'$ is an independent copy of $Z$. We therefore have $|M'/M|^2 \leq |M''/M|$, and using that $\delta|t| = \delta/\sg\leq2R$, we have
$$|M''/M|  \leq \delta^2\E_{Z'}\lb (\|W\| + 2R)^2\rb .$$
Now, $|p'(\xi)|^2 \leq \delta^2(\|Z\| + \delta|t|)^2\leq \delta^2\E_{Z'}\lb \lp\|W\| + 2R\rp^2\rb$ and $|p''(\xi)| \leq \delta^2$. Combining these estimates we obtain
\beq\label{fprimeprime}|f''(\xi)| \leq \delta^2 + 6\delta^2\E_{Z'}\lb \lp\|W\| + 2R\rp^2\rb \leq \delta^2(C  + \|Z\|^2).\eeq Taking the expectation of both sides of the inequality gives
$$\big|\E_{Z,\bm\chi'}[f''(\xi)]\big| \leq C\delta^2.$$ Substituting these calculations into \eqref{EYft} and taking $t=1/\sg$, we obtain
$$\E_Y\lb w_\ta(Y, \chi)\rb = \E_{Z,\bm\chi'}[f(1/\sg, Z, \bm\chi')] = 1 + \epsilon_1(\chi, \ta, \ta_*),$$ where $$|\epsilon_1(\chi, \ta, \ta_*)| \leq \frac12\sg^{-2}\big|\E_{Z,\bm\chi'}[f''(\xi)]\big| \leq C\lp\frac{\delta}{\sg}\rp^2 \leq C\lp\frac{\|\ta\|_\infty\vee 1}{\sg}\rp^2.$$
\end{proof}
\begin{lemma}
For all $\ta$ such that $(\|\ta\|_\infty\vee 1)/\sg\leq R$ and for all $\chi\in X$, we have
$$\E_Y\lb w_\ta(Y,\chi)Y\rb = T_1(\ta_*) + \ta_\chi - T_1(\ta) + \epsilon_2(\chi,\ta, \ta_*),$$ where
$$\|\epsilon_2(\chi,\ta, \ta_*)\|\leq C\frac{(\|\ta\|_\infty\vee 1)^2}{\sg},$$ and $C$ depends on $R$ and $d$ only. 
\end{lemma}
\begin{proof}
Recall from the previous Lemma that $\delta = \sup_{\chi,\chi'\in X}\|\ta_\chi - \ta_{*\chi'}\|$, so that $\delta/\sg \leq2R$. Also, we defined $t = 1/\sg$ and expressed $w_\ta(Y,\chi)$ for fixed $\chi$ and $\ta$ as $f(t, Z, \bm\chi')$. We Taylor expanded $f$ around $t=0$ to second order as:
\beqs
f(t, Z, \bm\chi')  &= 1 + tf'(0) + \frac12t^2f''(\xi) = 1 + t Z^T(v- \E_\chi[v]) + \frac12t^2f''(\xi)\\
&= 1 + tZ^T(\ta_\chi - T_1(\ta))+\frac12t^2f''(\xi)
\eeqs where $|f''(\xi)|\leq \delta^2(C + \|Z\|^2)$ for an absolute constant $C$. We now multiply this Taylor expansion by $Y = (1/t)Z + \ta_{*\bm\chi'}$, and take its $Z, \bm\chi'$-expectation:
\beqs \E_{Z, \bm\chi'}\lb(\frac1tZ + \ta_{*\bm\chi'})f(t, Z, \bm\chi')\rb =& \E_{Z, \bm\chi'}\bigg[\frac1tZ + \ta_{*\bm\chi'}+ ZZ^T(\ta_\chi - T_1(\ta)) \\
&+ tZ^T(\ta_\chi - T_1(\ta))\ta_{*\bm\chi'} + \frac12tf''(\xi)Z + \ta_{*\bm\chi'}\frac12t^2f''(\xi)\bigg]\\
&= T_1(\ta_*) + \ta_\chi - T_1(\ta) + \epsilon_2(\chi, \ta, \ta_*),\eeqs where
\beq
\epsilon_2(\chi, \ta, \ta_*) = \frac12t\E_{Z, \bm\chi'}\lb f''(\xi)Z\rb + \frac12t^2\E_{Z, \bm\chi'}\lb f''(\xi)\ta_{*\bm\chi'}\rb.
\eeq
Now, using $t=1/\sg,\delta/\sg\leq 2R,$ and $\delta\leq 2(\|\ta\|_\infty\vee1)$, we have
$$
\big\|\frac12t\E_{Z, \bm\chi'}\lb f''(\xi)Z\rb\big\| \leq \frac12(\delta^2/\sg)\E_Z \big[\|Z\|(C + \|Z\|^2)\big] \leq C(\|\ta\|_\infty\vee1)^2/\sg
$$
and
$$\big\|\frac12t^2\E_{Z, \bm\chi'}\lb f''(\xi)\ta_{*\bm\chi'}\rb\big\|\leq (\delta/\sg)^2\E_Z\lb C + \|Z\|^2\rb \leq C(\|\ta\|_\infty\vee1)/\sg.$$ Adding the two upper bounds together, we have
$$\|\epsilon_2(\chi,\ta,\ta_*)\|\leq C(\|\ta\|_\infty\vee1)^2/\sg,$$ as desired. 
\end{proof}
\begin{prop}\label{app:T1-EM}
For all $\ta$ such that $(\|\ta\|_\infty\vee 1)/\sg\leq R$, we have
$$\frac{\|T_1(G(\ta))-T_1(\ta_*)\|}{\|\ta\|_\infty\vee1}\leq C(\|\ta\|_\infty\vee1)/\sg$$ for some constant $C$ that depends on $d$ and $R$ only. 
\end{prop}
\begin{proof}
Recall that $G_\chi(\ta) = \E_Y\lb w_\ta(Y,\chi)Y\rb\big/\E_Y\lb w_\ta(Y,\chi)\rb.$ From the previous two lemmas, we have $\E_Y\lb w_\ta(Y,\chi)\rb = 1 + \epsilon_1(\chi),$ and $\E_Y\lb w_\ta(Y,\chi)Y\rb = T_1^* + \ta_\chi - T_1(\ta) + \epsilon_2(\chi)$. We then have
\beq\label{Gchi}G_{\chi}(\ta) = \frac{T_1^* + \ta_\chi - T_1(\ta) + \epsilon_2}{1 + \epsilon_1} = T_1^* + \ta_\chi - T_1(\ta)  + \epsilon_3(\chi),\eeq where 
$$\epsilon_3 = \frac{\epsilon_2 + (T_1(\ta_*)+ \ta_\chi - T_1(\ta))\epsilon_1}{1+\epsilon_1}.$$
Using the bounds on $\epsilon_1,\epsilon_2$ from the lemmas, we have
\beqs\label{eps3}\|\epsilon_3\|&\leq \|\epsilon_2\| + |\epsilon_1|\|T_1(\ta_*) +\ta_\chi - T_1(\ta)\|\\
& \leq \|\epsilon_2\| + C|\epsilon_1|(\|\ta\|_\infty \vee 1)\\
&\leq C(\|\ta\|_\infty \vee 1)^2/\sg.\eeqs
Averaging \eqref{Gchi} over $\chi$ then gives
$$T_1(G(\ta)) = \E_{\chi\sim\rho}G_\chi(\ta) = T_1^* + T_1 - T_1 + \E_\chi\epsilon_3(\chi) = T_1^* + \E_\chi\epsilon_3(\chi).$$ But the bound \eqref{eps3} on $\epsilon_3(\chi)$ is independent of $\chi$, so we have
$$\|T_1(G(\ta))-T_1^*\|\leq \E_\chi\|\epsilon_3(\chi)\| \leq C(\|\ta\|_\infty \vee 1)^2/\sg,$$ as desired. 
\end{proof}

\section{Miscellany}\label{app:miscellany}
\begin{lemma}
Let $f:\R^K\to \R$ and $\mathcal M\subset \R^K$ be a smooth function and manifold, respectively, where $\mathcal M$ is defined as the intersection of the level surfaces of functions $g_1, \dots, g_n$. Then $x$ is a critical point of $f\vert_{\mathcal M}$ iff there exist $\lambda_1, \dots, \lambda_n\in\R$ such that
\beq\label{app:critpoint}\nabla f(x) = \sum_{j=1}^n\lambda_j\nabla g_j(x).\eeq Moreover, a critical point $x$ is a saddle, local minimum, or local maximum of $f\vert_{\mathcal M}$ iff the quadratic form
\beq\label{app:critpoint2} \nabla^2 f(x) - \sum_{j=1}^n\lambda_j\nabla^2g_j(x)\eeq is indeterminate, positive definite, or negative definite, respectively, on the tangent plane to $\mathcal M$ at $x$. 
\end{lemma}
\begin{proof}
The first condition is standard. To show the second condition, note that a critical point $x$ is a local minimum (maximum) of $f\vert_{\mathcal M}$ if and only if for every curve $u(t)\subset\mathcal M$ going through $x$, the function $t\mapsto f(u(t))$ has a local minimum (maximum) at $t=t_x$, the point for which $u(t)=x$. 

Let $u$ be such a curve. Now, $(f\circ u)'(t_x) = \la\nabla f(x),  u'(t_x)\ra = 0$, since $\nabla f(x)$ lies in the span of $\nabla g_j(x),j=1,\dots, n$ while $u'(t_x)$ lies in the tangent plane to $\mathcal M$ at $x$. Hence $t\mapsto f(u(t))$ has a critical point at $t_x$. Note that
\beqs
\frac{d^2}{dt^2}f(u(t))\big\vert_{t=t_x}&= \la \nabla f(x), u''(t_x)\ra + u'(t_x)^T\nabla^2f(x)u'(t_x) \\
&=  \sum_{j=1}^n\lambda_j\la \nabla g_j(x), u''(t_x)\ra+ u'(t_x)^T\nabla^2f(x)u'(t_x) \\
&= u'(t_x)^T\lp\nabla^2 f(x) - \sum_{j=1}^n\lambda_j\nabla^2g_j(x)\rp u'(t_x).
\eeqs
The last line follows from the fact that 
$$0 = \frac{d^2}{d\ta^2}g_j(u(t_x)) =\la \nabla g_j(x), u''(t_x)\ra + u'(t_x)^T\nabla^2g_j(x)u'(t_x),\;j=1,\dots, n.$$
Now, $u'(t_x)$ is any vector in the tangent plane of $\mathcal M$ at $x$, i.e. any vector perpendicular to $\nabla g_j(x), j=1, \dots, n$. Thus, for $\frac{d^2}{dt^2}f(u(t_x))$ to have the same sign for all curves $u$, the quadratic form $\nabla^2 f(x) - \sum_{j=1}^n\lambda_j\nabla^2g_j(x)$ must be determinate on the tangent plane at $x$. 
\end{proof}

\end{document}